\documentclass[12pt, a4paper]{amsart}

\usepackage{fullpage}
\usepackage{amsmath,amssymb}
\usepackage[colorlinks=true]{hyperref}
\usepackage{xypic}

\newcommand{\Z}{\mathbb{Z}} 
\newcommand{\Q}{\mathbb{Q}}

\newcommand{\F}{\mathbb{F}}
 
\theoremstyle{plain}
\newtheorem{theorem}{Theorem}[section]
\newtheorem{lemma}[theorem]{Lemma}
\newtheorem{corollary}[theorem]{Corollary}
\newtheorem{prop}[theorem]{Proposition}
\newtheorem{prop-def}[theorem]{Proposition / Definition}
\newtheorem{sec:hyp}[theorem]{Hypothesis}
\newtheorem{algorithm}[theorem]{Algorithm}

\newtheorem*{theorem*}{Theorem}

\theoremstyle{remark}

\newtheorem{remark}[theorem]{Remark}

\newtheorem{example}[theorem]{Example}

\newtheorem*{note*}{Note}
\newtheorem*{remark*}{Remark}
\newtheorem*{example*}{Example}

\theoremstyle{definition}
\newtheorem*{definition*}{Definition}
\newtheorem*{hypothesis*}{Hypothesis}
\newtheorem*{assumptions*}{Assumptions}

\newcommand{\Gal}{\mathrm{Gal}}

\newcommand{\disc}{\mathrm{Disc}}

\newcommand{\Cl}{\mathrm{Cl}}

\newcommand{\Aut}{\mathrm{Aut}}
\newcommand{\End}{\mathrm{End}}
\newcommand{\Mat}{\mathrm{Mat}}
\newcommand{\Hom}{\mathrm{Hom}}
\newcommand{\Ext}{\mathrm{Ext}}
\newcommand{\GL}{\mathrm{GL}}

\newcommand{\op}{\mathrm{op}}

\newcommand{\calO}{\mathcal{O}}

\newcommand{\Out}{\mathrm{Out}}
\newcommand{\Inn}{\mathrm{Inn}}

\newcommand{\OL}{{\calO_L}}

\newcommand{\sseq}{\subseteq}

\newcommand*{\rright}[1]{\multicolumn{1}{c|}{#1}}
\newcommand*{\lleft}[1]{\multicolumn{1}{|c}{#1}}

\title[Computing isomorphisms between lattices]{Computing isomorphisms between lattices}

\author{Tommy Hofmann}
\address{
Fachbereich Mathematik\\
Technische Universit\"at Kaiserslautern\\
67663 Kaiserslautern\\
Germany}
\email{thofmann@mathematik.uni-kl.de}
\urladdr{http://www.mathematik.uni-kl.de/$\sim$thofmann}

\author{Henri Johnston}
\address{
Department of Mathematics\\
University of Exeter\\
Exeter\\
EX4 4QF\\
United Kingdom
}
\email{H.Johnston@exeter.ac.uk}
\urladdr{http://emps.exeter.ac.uk/mathematics/staff/hj241}

\date{Version of 2nd March 2020}

\begin{document}

\begin{abstract}
Let $K$ be a number field, let $A$ be a finite-dimensional semisimple $K$-algebra and let $\Lambda$ be an $\mathcal{O}_{K}$-order in $A$. 
It was shown in previous work that, under certain hypotheses on $A$, 
there exists an algorithm that for a given (left) $\Lambda$-lattice
$X$ either computes a free basis of $X$ over $\Lambda$ or shows that $X$ is not free over $\Lambda$.
In the present article, we generalize this by showing that, under weaker hypotheses on $A$,
there exists an algorithm that for two given $\Lambda$-lattices $X$ and $Y$
either computes an isomorphism $X \rightarrow Y$ or determines that $X$ and $Y$ are not isomorphic.
The algorithm is implemented in \textsc{Magma} 
for $A=\Q[G]$, $\Lambda=\Z[G]$ and $\Lambda$-lattices $X$ and $Y$ contained in $\Q[G]$,
where $G$ is a finite group satisfying certain hypotheses.
This is used to investigate the Galois module structure of rings of integers and ambiguous ideals of tamely ramified
Galois extensions of $\Q$ with Galois group isomorphic to $Q_{8} \times C_{2}$, the direct product of the
quaternion group of order $8$ and the cyclic group of order $2$.
\end{abstract}

\maketitle

\section{Introduction}
Let $K$ be a number field with ring of integers $\mathcal{O}_{K}$.
Let $A$ be a finite-dimensional semisimple $K$-algebra and let $\Lambda$ be an $\mathcal{O}_{K}$-order in $A$.
For example, if $G$ is a finite group then the group ring $\mathcal{O}_{K}[G]$ is an order in the group algebra $K[G]$.
A $\Lambda$-lattice is a (left) $\Lambda$-module that is finitely generated and torsion-free over $\mathcal{O}_{K}$. 
In previous work \cite{Bley2008,Bley2011} of Bley and the second named author,
it was shown that, under certain hypotheses on $A$, there exists an algorithm that for a given $\Lambda$-lattice
$X$ either computes a free basis of $X$ over $\Lambda$ or shows that $X$ is not free over $\Lambda$.
In the present article, we generalize this by showing that, under the (weaker) hypotheses on $A$ discussed below,
there exists an algorithm that for two given $\Lambda$-lattices $X$ and $Y$
either computes an isomorphism $X \rightarrow Y$ or determines that $X$ and $Y$ are not isomorphic.

The key theoretical results of the present article are necessary and sufficient conditions for the existence
of an isomorphism $X \rightarrow Y$ between two given $\Lambda$-lattices (see \S \ref{sec:conditions-iso}).
One of these criteria forms the basis of the main algorithm (Algorithm \ref{alg:find-iso}).

Let $A=A_{1} \oplus \cdots \oplus A_{r}$ be the decomposition of $A$ into indecomposable two-sided ideals
and let $K_{i}$ denote the center of the simple algebra $A_{i}$.
A key step of Algorithm \ref{alg:find-iso} requires the following two hypotheses, which are discussed in detail
in \S \ref{sec:hyp}.

\begin{itemize}
\item[(H1)]
For each $i$, we can compute an explicit isomorphism $A_{i} \cong \Mat_{n_{i} \times n_{i}}(D_{i})$ of $K$-algebras,
where $D_{i}$ is a skew field with center $K_{i}$.
\item[(H2)] For each $i$, every maximal $\mathcal{O}_{K}$-order $\Delta_{i}$ in $D_{i}$ has the following properties:
\begin{enumerate}
\item we can solve the principal ideal problem for fractional left $\Delta_{i}$-ideals, and
\item $\Delta_{i}$ has the locally free cancellation property.
\end{enumerate}
\end{itemize}
The key step in question is the computation of isomorphisms of certain lattices over maximal orders in each simple component 
$A_{i}$ (see \S \ref{sec:isom-over-max}). 
Two other crucial steps are the computation of endomorphism rings (a method for the more general problem of computing homomorphism groups is given in \S \ref{sec:saturation-homgroups}) and isomorphism testing for localized lattices (see \S \ref{sec:lociso}). 
An ad hoc method for reducing the number of tests required in the last (and most expensive) step of Algorithm \ref{alg:find-iso}
is outlined in \S \ref{sec:cut-number-of-tests}. 

An important motivation for this work is the investigation of the Galois module structure of rings of integers and their ambiguous ideals.
Let $L/K$ be a finite Galois extension of number fields and let $G=\Gal(L/K)$.
The classical Normal Basis Theorem says that $L$ is free of rank $1$ as a module over the group algebra $K[G]$.
A much more difficult problem is that of determining the structure of the ring of integers $\mathcal{O}_{L}$
over an order such as $\mathcal{O}_{K}[G]$.
More generally, one can consider the structure of ambiguous ideals of $\mathcal{O}_{L}$. 
There is a very large body of work on these problems and we now mention only a small selection of results.

By far the most progress has been made in the case that $L/K$ is (at most) tamely ramified.
In this setting, it is well known that
$\mathcal{O}_{L}$ is a locally free $\mathcal{O}_{K}[G]$-lattice of rank $1$
(see \cite{Noether1932} or \cite[I, \S 3]{MR717033}). 
Thus one can consider the class $(\mathcal{O}_{L})$ in the locally free class group $\Cl(\mathcal{O}_{K}[G])$.
If $G$ is abelian or of odd order then $(\mathcal{O}_{L})$ determines $\mathcal{O}_{L}$ up to isomorphism, but for $K=\Q$ and certain non-abelian $G$ of even order, this class only determines $\mathcal{O}_{L}$ up to stable isomorphism. 
If we restrict scalars and consider $\mathcal{O}_{L}$ as a $\Z[G]$-module then we obtain a class $(\mathcal{O}_{L})_{\Z[G]}$
in $\Cl(\Z[G])$. The root number class $W_{L/K}$ in $\Cl(\Z[G])$ was defined by Ph.\ Cassou-Nogu\`es in terms 
of Artin root numbers of the irreducible symplectic characters of $G$, and Taylor's proof \cite{MR608528} of a conjecture of Fr\"ohlich
shows that $(\mathcal{O}_{L})_{\Z[G]}=W_{L/K}$ (see \cite[I]{MR717033} for an overview).
In particular, if $G$ is abelian or of odd order then $W_{L/K}$ is trivial and $\mathcal{O}_{L}$ is always free over $\Z[G]$ 
(unfortunately, this approach does not give a description of a $\Z[G]$-basis).
There are many examples of interesting behaviour when $G \cong Q_{4n}$, the quaternion group of order $4n$. 
For instance, when $K=\Q$ and $G \cong Q_{8}$, both possibilities for $W_{L/K}$ in $\Cl(\Z[Q_{8}]) \cong \{ \pm 1 \}$
occur infinitely often.
Moreover, Cougnard \cite{cougnard-stable} gave an example with $K=\Q$ and $G \cong Q_{32}$
such that $\mathcal{O}_{L}$ is stably free but not free over $\Z[G]$. 

If $L/K$ is wildly ramified then the situation becomes much more difficult because $\mathcal{O}_{L}$
is not even locally free over $\Z[G]$. One approach is to consider the structure of
$\mathcal{O}_{L}$ over the so-called associated order 
$\mathfrak{A}_{L/K}= \{ \lambda \in K[G] \mid \lambda \mathcal{O}_{L} \subseteq \mathcal{O}_{L} \}$, which is equal to 
$\mathcal{O}_{K}[G]$ if and only if $L/K$ is tamely ramified.
An important result in this setting is Leopoldt's Theorem \cite{leopoldt}, which
says that for any finite abelian extension $L/\Q$ the ring of integers $\mathcal{O}_{L}$
is always free over $\mathfrak{A}_{L/\Q}$ and, in addition, gives an explicit construction of a free generator
(also see \cite{MR1037435}). However, in general,  
$\mathcal{O}_{L}$ need not even be locally free over $\mathfrak{A}_{L/K}$. 
An alternative approach is to consider Chinburg's invariant $\Omega(L/K,2)$ in $\Cl(\Z[G])$.
This is equal to $(\mathcal{O}_{L})_{\Z[G]}$ in the tamely ramified case, but is also defined when $L/K$ is wildly ramified. 
Fr\"ohlich \cite{MR507603} generalized the definition of $W_{L/K}$ to the wildly ramified case.
Chinburg's `$\Omega(2)$ conjecture'  \cite{MR786352} asserts that $\Omega(L/K,2)=W_{L/K}$;
in the tamely ramified case this is equivalent to the aforementioned theorem of Taylor. 

Another interesting problem is the determination of the structure of the square root $\mathcal{A}_{L/K}$
of the inverse different $\mathcal{D}_{L/K}^{-1}$.  The former exists (and is unique) precisely when the latter is a square, which 
can be tested using Hilbert's formula \cite[IV, Proposition 4]{MR554237}. 
In particular, $\mathcal{A}_{L/K}$ always exists when $|G|$ is odd. 
Erez  showed that if $\mathcal{A}_{L/K}$ exists then $\mathcal{A}_{L/K}$ is locally free over $\Z[G]$
if and only if $L/K$ is weakly ramified, that is, the second ramification group of every prime is trivial 
(see \cite{MR1128708} and \cite[footnote 1]{MR3550860}). 
He also showed that if $|G|$ is odd and $L/K$ is tamely ramified then $\mathcal{A}_{L/K}$ is free over $\Z[G]$.
This result was generalized by Caputo and Vinatier \cite{MR3550860} and has been further extended in recent work of
Agboola and Caputo \cite{agboola-caputo-preprint} who showed that if $L/K$ is tamely ramified
and $\mathcal{A}_{L/K}$ exists then $\Omega(L/K,2)=(\mathcal{O}_{L})_{\Z[G]}=(\mathcal{A}_{L/K})_{\Z[G]}$ in $\Cl(\Z[G])$.
Moreover, Erez has asked whether $\Omega(L/K,2)=(\mathcal{A}_{L/K})_{\Z[G]}$ for every weakly ramified extension $L/K$
such that $\mathcal{A}_{L/K}$ exists (see \cite[Question 2]{MR3550860}); under the assumption of Chinburg's $\Omega(2)$ conjecture this may been seen as a generalization of a question of Vinatier \cite{MR2019023}.

We now review previous work on related algorithms.
Let $K$ be a number field, let $A$ be a finite-dimensional semisimple $K$-algebra and let 
$\Lambda$ be an $\mathcal{O}_{K}$-order in $A$.
In the case that $A$ is commutative and $X$ is a $\Lambda$-lattice such that 
$K \otimes_{\mathcal{O}_{K}} X$ is free of rank $1$ over $A$, Bley \cite{Bley1997} gave algorithms that determine whether $X$ is locally free over $\Lambda$,
and either explicitly construct a free generator for $X$ over $\Lambda$ or show that no such generator exists
(though stated for a specific arithmetic case, it is straightforward to see that the algorithms work in this more general setting).
A noncommutative higher-rank generalization was given by Bley and the second named author \cite{Bley2008,Bley2011}, and it is this work that is in turn generalized in the present article.
In \cite{bley-endres}, Bley and Endres again considered the case in which $A$ is commutative and 
gave algorithms for computing the Picard group
$\mathrm{Pic}(\Lambda)$ and solving the corresponding refined discrete logarithm problem
(and thus for computing isomorphisms between invertible $\Lambda$-submodules of $A$).
In the case that $A=K[G]$ for some finite group $G$, 
Bley and Wilson \cite{Bley2009} gave algorithms for computing the relative algebraic $K$-group
$K_{0}(\mathcal{O}_{K}[G],K)$ and solving the discrete logarithm problem in both $K_{0}(\mathcal{O}_{K}[G]),K)$ and the locally free class group $\Cl(\mathcal{O}_{K}[G])$
(note that this is a lot weaker than the \emph{refined} discrete logarithm problem and can only be used to determine whether two locally free $\mathcal{O}_{K}[G]$-lattices are stably isomorphic). 

The main algorithm of the present article (Algorithm \ref{alg:find-iso}) is very general: it is only subject to the hypotheses (H1) and (H2)
and does not require the order $\Lambda$
to be commutative or the lattices to be locally free, for example. Moreover, not only does it determine whether two lattices are isomorphic
(rather than just stably isomorphic), but it also explicitly computes an isomorphism, if one exists. 

In \S \ref{sec:exp}, we discuss experimental results obtained by using a proof of concept implementation of 
Algorithm \ref{alg:find-iso} in \textsc{Magma}~\cite{Wieb1997} for $A=\Q[G]$, $\Lambda=\Z[G]$ and $\Lambda$-lattices $X$ and $Y$ contained in $\Q[G]$, where $G$ is a finite group satisfying certain hypotheses.
We now fix $G=Q_{8} \times C_{2}$,
the direct product of the quaternion group of order $8$ and the cyclic group of order $2$,
and remark that this satisfies the hypotheses required by the implementation.
Moreover, Swan \cite{Swan1983} showed there exist $\Z[G]$-lattices that are stably free but not free, 
but that for any group $H$ with $|H|<16$, every stably free $\Z[H]$-lattice is in fact free.
Using Swan's results, Cougnard \cite{cougnard-H8-C2} gave examples of tamely ramified Galois extensions $L/\Q$ with $\Gal(L/\Q) \cong G$ such that 
$\mathcal{O}_{L}$ is stably free but not free over $\Z[\Gal(L/\Q)]$.
Using the implementation of Algorithm \ref{alg:find-iso}, we find the extension of smallest absolute discriminant with this property.
For fixed tamely ramified Galois extensions $L/\Q$ with $\Gal(L/\Q) \cong G$,
we then examine the distribution of isomorphism classes of ambiguous ideals of $\mathcal{O}_{L}$. 

\subsection*{Acknowledgments}
The authors are indebted to Werner Bley for suggesting this project and for his help in drafting an early version of \S \ref{subsec:res-iso-max}.
The authors also wish to thank: 
Alex Bartel for pointing out a mistake in an earlier version of this article;
Nigel Byott for several helpful conversations and comments; 
Gunter Malle for numerous helpful comments; and the referee for a thorough and thoughtful report.
The first named author was supported by Project II.2 of SFB-TRR 195 `Symbolic Tools in
Mathematics and their Application'
of the German Research Foundation (DFG).
The second named author was supported by EPSRC First Grant EP/N005716/1 `Equivariant Conjectures in Arithmetic'.

\section{Preliminaries on lattices and orders}\label{sec:prelims-on-lattices-and-orders}

For further background, we refer the reader to \cite{curtisandreiner_vol1,Reiner2003}.
Let $\mathcal{O}$ be a Dedekind domain with field of fractions $K$.
To avoid trivialities, we assume that $\mathcal{O} \neq K$.

\subsection{Lattices over Dedekind domains}
An $\mathcal{O}$-lattice $M$ is a finitely generated torsion-free $\mathcal{O}$-module, or equivalently,
a finitely generated projective $\mathcal{O}$-module.
For a prime ideal $\mathfrak{p}$ of $\mathcal{O}$, we
let $\mathcal{O}_{\mathfrak{p}}$ denote the localization (not completion) of $\mathcal{O}$ at $\mathfrak{p}$.
We define the localization of $M$ at $\mathfrak{p}$ to be 
the $\mathcal{O}_{\mathfrak{p}}$-lattice $M_{\mathfrak{p}} := \mathcal{O}_{\mathfrak{p}} \otimes_{\mathcal{O}} M$.
We can identify $M$ with its image $1 \otimes M$ in $\mathcal{O}_{\mathfrak{p}} \otimes_{\mathcal{O}} M$ and thus
may view $M$ as embedded in $M_{\mathfrak{p}}$. 
In particular, we may view $M$ as a finitely generated
$\mathcal{O}$-submodule of the finite-dimensional $K$-vector space $K \otimes_{\mathcal{O}} M$,
which we may write in the simpler form
\[
KM := \{ \alpha_{1} m_{1} + \cdots + \alpha_{r} m_{r} \mid r \in \Z_{\geq 0}, \alpha_{i} \in K, m_{i} \in M \}.
\]
Then for any choice of $\mathfrak{p}$, we may identify $M_{\mathfrak{p}}$ with
\[
\mathcal{O}_{\mathfrak{p}}M := 
\{ \alpha_{1} m_{1} + \cdots + \alpha_{r} m_{r} \mid r \in \Z_{\geq 0}, \alpha_{i} \in \mathcal{O}_{\mathfrak{p}}, m_{i} \in M \}.
\]

Alternatively, given a finite-dimensional $K$-vector space $V$, we can define an $\mathcal{O}$-lattice $M$
in $V$ to be a finitely generated $\mathcal{O}$-submodule of $V$. 
Then $KM$ as defined above is a $K$-vector subspace of $V$ and we say that $M$ is a full $\mathcal{O}$-lattice in $V$ if $K M=V$.
Moreover, for each $\mathfrak{p}$ we define $M_{\mathfrak{p}}$ to be $\mathcal{O}_{\mathfrak{p}}M$ as above and note that this is
an $\mathcal{O}_{\mathfrak{p}}$-lattice in $V$. We shall switch between the two ways of viewing $\mathcal{O}$-lattices as convenient.

\subsection{Lattices over orders}
Let $A$ be a separable $K$-algebra, that is, a finite-dimensional semisimple $K$-algebra
such that the center of each simple component of $A$ is a separable field extension of $K$.
An $\mathcal{O}$-order in $A$ is a subring $\Lambda$ of $A$ (so in particular has the same unity element as $A$) such that
$\Lambda$ is a full $\mathcal{O}$-lattice in $A$.
Note that $\Lambda$ is both left and right noetherian, since $\Lambda$ is finitely generated over $\mathcal{O}$. 
A left $\Lambda$-lattice $X$ is a left $\Lambda$-module that is also an $\mathcal{O}$-lattice; in this case, $KX$ may be viewed as a left $A$-module.

Henceforth all modules (resp.\ lattices) shall be assumed to be left modules (resp.\ lattices) unless otherwise stated.
Two $\Lambda$-lattices are said to be isomorphic if they are isomorphic as $\Lambda$-modules.
For any maximal ideal $\mathfrak{p}$ of $\mathcal{O}$, the localization $\Lambda_{\mathfrak{p}}$ is an 
$\mathcal{O}_{\mathfrak{p}}$-order in $A$. Moreover, localizing a $\Lambda$-lattice $X$ yields
a $\Lambda_{\mathfrak{p}}$-lattice $X_{\mathfrak{p}}$.  
Two $\Lambda$-lattices $X$ and $Y$ are said to be locally isomorphic 
(or in the same genus) if and only if the
$\Lambda_{\mathfrak{p}}$-lattices $X_{\mathfrak{p}}$ and $Y_{\mathfrak{p}}$ are isomorphic
for all maximal ideals $\mathfrak{p}$ of $\mathcal{O}$. 

\begin{lemma}\label{lem:ext}
Let $\Lambda \subseteq \Lambda'$ be $\mathcal{O}$-orders in $A$ and 
let $\Gamma$ be either $\Lambda'$ or $\Lambda'_{\mathfrak{p}}$ for some maximal ideal $\mathfrak{p}$ of $\mathcal{O}$.
Let $f \colon X \to Y$ be a homomorphism of $\Lambda$-lattices.
Then the following hold.
\begin{enumerate}
\item The $\Gamma$-module $\Gamma X$ generated by $X$ is in fact a $\Gamma$-lattice. Similarly for $\Gamma Y$.
\item There exists a unique homomorphism of $A$-modules $f^{A} \colon KX \to KY$ extending $f$.
\item There exists a unique homomorphism of $\Gamma$-lattices $f^\Gamma \colon \Gamma X \to \Gamma Y$ extending $f$.
\item If $f$ is injective (resp.\ surjective), then $f^{A}$ and $f^{\Gamma}$ are injective (resp.\ surjective).
\end{enumerate}
\end{lemma}

\begin{proof}
This is straightforward. The key points are to
(a) check that $\Gamma X$ is in fact finitely generated over either $\mathcal{O}$ or $\mathcal{O}_{\mathfrak{p}}$ as appropriate;
(b) extend $f$ to $KX$ using $K$-linearity; (c) restrict $f^{A}$ to $\Gamma X$; (d) (injectivity) check that $\ker(f)$ is an $\mathcal{O}$-lattice of full rank in $\ker(f^{A})$; and (d) (surjectivity) use the definitions of $K Y$ and $\Gamma Y$.
\end{proof}

\section{Conditions for two lattices to be isomorphic}\label{sec:conditions-iso}

\subsection{Restricting isomorphisms over maximal orders}\label{subsec:res-iso-max}
Let $\mathcal{O}$ be a Dedekind domain with field of fractions $K$ and assume that $\mathcal{O} \neq K$.
Let $\Lambda$ be an $\mathcal{O}$-order in a separable $K$-algebra $A$.
By \cite[(10.4)]{Reiner2003} there exists a (not necessarily unique) maximal $\mathcal{O}$-order $\mathcal{M}$
such that $\Lambda \subseteq \mathcal{M} \subseteq A$.
It is clear that if $X$ and $Y$ are isomorphic $\Lambda$-lattices then they are locally isomorphic and 
$\mathcal M X$ and $\mathcal MY$ are isomorphic $\mathcal M$-lattices.
We now investigate the additional conditions under which a converse holds.
In doing so, we generalize the results of \cite[\S 2]{Bley2008},
where necessary and sufficient conditions were given for a $\Lambda$-lattice to be free.

\begin{theorem}\label{thm:restrict-from-iso-over-max-order}
Let $X$ and $Y$ be locally isomorphic $\Lambda$-lattices such that there exists
an isomorphism of $\mathcal{M}$-lattices $f \colon \mathcal{M} X \to \mathcal{M} Y$.
Then $f$ restricts to an isomorphism $f|_X \colon X \to Y$ of $\Lambda$-lattices if and only if $f(X) \subseteq Y$.
\end{theorem}

\begin{remark}\label{rmk:reverse-inclusion}
The condition $f(X) \subseteq Y$ can be replaced with $Y \subseteq f(X)$ by making obvious changes to the proof below.
The corollaries that follow can be rephrased analogously.
\end{remark}

\begin{proof}[Proof of Theorem \ref{thm:restrict-from-iso-over-max-order}]
One direction is trivial.
For the other, suppose that $f(X) \subseteq Y$. 
The restriction $f|_{X}$ is clearly an injective $\Lambda$-lattice homomorphism, so it only remains to show that $f(X)=Y$.

Let $\mathfrak{p}$ be a maximal ideal of $\calO$ and let $f_{\mathfrak{p}} \colon X_{\mathfrak{p}} \to Y_{\mathfrak{p}}$ be an isomorphism of $\Lambda_{\mathfrak{p}}$-lattices.
By Lemma~\ref{lem:ext} there exist unique $\mathcal{M}_{\mathfrak{p}}$-lattice isomorphisms
$f^{\mathcal{M}_{\mathfrak{p}}} \colon \mathcal{M}_{\mathfrak{p}} X \to \mathcal{M}_{\mathfrak{p}} Y$ and  $f_{\mathfrak{p}}^{\mathcal{M}_{\mathfrak{p}}} \colon \mathcal{M}_{\mathfrak{p}} X_{\mathfrak{p}} \to \mathcal{M}_{\mathfrak{p}} Y_{\mathfrak{p}}$ extending $f$ and $f_{\mathfrak{p}}$, respectively.
Note that $\mathcal{M}_{\mathfrak{p}}X = \mathcal{M}_{\mathfrak{p}}X_{\mathfrak{p}}$ and $\mathcal{M}_{\mathfrak{p}}Y = \mathcal{M}_{\mathfrak{p}}Y_{\mathfrak{p}}$
and so we can and do consider $f^{\mathcal{M}_{\mathfrak{p}}}$ as a map
$\mathcal{M}_{\mathfrak{p}} X_{\mathfrak{p}} \to \mathcal{M}_{\mathfrak{p}} Y_{\mathfrak{p}}$.

Observe that we have the following equalities and containment:
\[
f^{\mathcal{M}_{\mathfrak{p}}}_{\mathfrak{p}}(X_{\mathfrak{p}}) = f_{\mathfrak{p}}(X_{\mathfrak{p}}) = Y_{\mathfrak{p}} \supseteq f(X)_{\mathfrak{p}} = \mathcal{O}_{\mathfrak{p}} f(X) = \mathcal{O}_{\mathfrak{p}} f^{\mathcal{M}_{\mathfrak{p}}}(X) = f^{\mathcal{M}_{\mathfrak{p}}}(\mathcal{O}_{\mathfrak{p}} X) = f^{\mathcal{M}_{\mathfrak{p}}}(X_{\mathfrak{p}}).
\]
Thus if  $[ - : - ]_{\mathcal{O}_{\mathfrak{p}}}$ denotes the module index (see \cite[\S 3]{Frohlich1965} or \cite[II.4]{ft}) then we have
\begin{align*}
[ Y_{\mathfrak{p}} : f(X)_{\mathfrak{p}} ] _{\mathcal{O}_{\mathfrak{p}}}
&= [ Y_{\mathfrak{p}} : f^{\mathcal{M}_{\mathfrak{p}}}(X_{\mathfrak{p}}) ]_{\mathcal{O}_{\mathfrak{p}}}
= [ f_{\mathfrak{p}}^{\mathcal{M}_{\mathfrak{p}}}(X_{\mathfrak{p}}) : (f_{\mathfrak{p}}^{\mathcal{M}_{\mathfrak{p}}} \circ (f_{\mathfrak{p}}^{\mathcal{M}_{\mathfrak{p}}})^{-1} \circ f^{\mathcal{M}_{\mathfrak{p}}})(X_{\mathfrak{p}}) ]_{\mathcal{O}_{\mathfrak{p}}} \\
&= [ X_{\mathfrak{p}} : ((f_{\mathfrak{p}}^{\mathcal{M}_{\mathfrak{p}}})^{-1} \circ f^{\mathcal{M}_{\mathfrak{p}}})(X_{\mathfrak{p}}) ] _{\mathcal{O}_{\mathfrak{p}}}
= \mathrm{det}_{\mathcal{O}_{\mathfrak{p}}}((f_{\mathfrak{p}}^{\mathcal{M}_{\mathfrak{p}}})^{-1} \circ f^{\mathcal{M}_{\mathfrak{p}}})\mathcal{O}_{\mathfrak{p}}.
\end{align*}
Since $(f_{\mathfrak{p}}^{\mathcal{M}_{\mathfrak{p}}})^{-1} \circ f^{\mathcal{M}_{\mathfrak{p}}} : \mathcal{M}_{\mathfrak{p}} X_{\mathfrak{p}} \to \mathcal{M}_{\mathfrak{p}} X_{\mathfrak{p}}$
is an automorphism of $\mathcal{M}_{\mathfrak{p}}$-lattices and thus of $\mathcal{O}_{\mathfrak{p}}$-lattices, we must have 
$\mathrm{det}_{\mathcal{O}_{\mathfrak{p}}}((f_{\mathfrak{p}}^{\mathcal{M}_{\mathfrak{p}}})^{-1} \circ f^{\mathcal{M}_{\mathfrak{p}}}) \in \mathcal{O}_{\mathfrak{p}}^{\times}$
and hence $[ Y_{\mathfrak{p}} : f(X)_{\mathfrak{p}} ] _{\mathcal{O}_{\mathfrak{p}}} = \mathcal{O}_{\mathfrak{p}}$.
Since this holds for every maximal ideal $\mathfrak{p}$ of $\calO$ and $f(X) \subseteq Y$ by assumption, we must have $f(X) = Y$.
Therefore $f|_{X} \colon X \to Y$ is an isomorphism of $\Lambda$-lattices.
\end{proof}

For an $\mathcal{M}$-lattice $N$, let $\End_{\mathcal{M}}(N)$ denote the ring of
$\mathcal{M}$-endomorphisms of $N$, and let $\Aut_{\mathcal{M}}(N)$ be the group of 
$\mathcal{M}$-automorphisms of $N$. 
Note that $\Aut_{\mathcal{M}}(N) = \End_{\mathcal{M}}(N)^{\times}$.

\begin{corollary}\label{cor:nonalg}
Two $\Lambda$-lattices $X$ and $Y$ are isomorphic if and only if
\begin{enumerate}
\item $X$ and $Y$ are locally isomorphic,
\item there exists an isomorphism of $\mathcal{M}$-lattices
$f \colon \mathcal{M}X \rightarrow \mathcal{M}Y$, and 
\item there exists an automorphism
$g \in \Aut_{\mathcal{M}}(\mathcal{M}Y)$ such that $(g \circ f)(X) \subseteq Y$.
\end{enumerate}
Further, when this is the case, an isomorphism is given by $(g \circ f) \colon X \rightarrow Y$.
\end{corollary}

\begin{proof}
If (a), (b) and (c) hold then $X$ and $Y$ are isomorphic by Theorem \ref{thm:restrict-from-iso-over-max-order}.
Suppose conversely that there exists a $\Lambda$-lattice isomorphism $h : X \rightarrow Y$.
Then (a) clearly holds, (b) holds with $f:= h^{\mathcal{M}}$ by Lemma \ref{lem:ext} and (c) holds with
$g:=\mathrm{Id}_{\mathcal{M}Y}$ (the identity map on $\mathcal{M}Y$).
\end{proof}

Most of the following notation is adopted from \cite{bley-boltje}.
Denote the center of a ring $R$ by $Z(R)$. 
Set $C=Z(A)$ and let $\mathcal{O}_{C}$ be the
integral closure of $\mathcal{O}$ in $C$.
Let $e_{1}, \ldots, e_{r}$ be the primitive idempotents
of $C$ and set $A_{i} = e_{i}A$. Then
\begin{equation}\label{eq:weddecomp}
A = A_{1} \oplus \cdots \oplus A_{r}
\end{equation}
is a decomposition of $A$ into indecomposable two-sided ideals (see \cite[(3.22)]{curtisandreiner_vol1}).
Each $A_{i}$ is a simple $K$-algebra with identity element $e_{i}$.
The centers $K_{i} := Z(A_{i})$ are finite field extensions of $K$ via $K \rightarrow K_{i}$, $\alpha \mapsto e_{i} \alpha$,
and we have $K$-algebra isomorphisms $A_{i} \cong \Mat_{n_{i} \times n_{i}}(D_{i})$ where $D_{i}$ is a skew field with $Z(D_{i}) \cong K_{i}$ 
(see \cite[(3.28)]{curtisandreiner_vol1}).
The Wedderburn decomposition (\ref{eq:weddecomp}) induces decompositions
\begin{equation}\label{eq:central-idem-decomps}
C = K_{1} \oplus \cdots \oplus K_{r}, 
\quad \mathcal{O}_{C} = \mathcal{O}_{K_{1}} \oplus \cdots \oplus \mathcal{O}_{K_{r}}
\quad  \textrm{and} \quad \mathcal{M} =  \mathcal{M}_{1} \oplus \cdots \oplus \mathcal{M}_{r},
\end{equation}
where we have set $\mathcal{M}_{i} = e_{i}\mathcal{M}$. 
By \cite[(10.5)]{Reiner2003} each $\mathcal{M}_{i}$ is a maximal $\mathcal{O}$-order (and thus a maximal $\mathcal{O}_{K_{i}}$-order) in $A_{i}$.

For an $\mathcal{M}$-lattice $N$ in a $K$-vector space $V$,
we set $V_{i} = e_{i}V$ and $\mathcal{M}_{i}N = e_{i}(\mathcal{M}N)$.
Then $\mathcal{M}_{i}N$ is a full $\mathcal{M}_{i}$-lattice in $V_{i}$. 
The decomposition \eqref{eq:central-idem-decomps} in turn induces decompositions
\begin{eqnarray}
\End_{\mathcal{M}}(N) &=& \End_{\mathcal{M}_{1}}(\mathcal{M}_{1}N) 
\oplus \dotsb \oplus \End_{\mathcal{M}_{r}}(\mathcal{M}_{r}N) 
\text{ and } \\
\Aut_{\mathcal{M}}(N) &=& \Aut_{\mathcal{M}_{1}}(\mathcal{M}_{1}N) \times 
\dotsb \times \Aut_{\mathcal{M}_{r}}(\mathcal{M}_{r}N).
\end{eqnarray}

\begin{corollary}\label{cor:ns-breakdown}
Two $\Lambda$-lattices $X$ and $Y$ are isomorphic if and only if
\begin{enumerate}
\item $X$ and $Y$ are locally isomorphic,
\item there exist isomorphisms of $\mathcal{M}_{i}$-lattices $f_{i} \colon \mathcal{M}_{i}X \rightarrow \mathcal{M}_{i}Y$ for each $i$, and 
\item there exist
$g_{i} \in \Aut_{\mathcal{M}_{i}}(\mathcal{M}_{i}Y)$ for each $i$ such that $
(\sum_{i=1}^{r}(g_{i} \circ f_{i}))(X) \subseteq Y$.
\end{enumerate}
Further, when this is the case, an isomorphism is given by $(\sum_{i=1}^{r}(g_{i} \circ f_{i})) \colon X \to Y$.
\end{corollary}

Now let $Y$ be a $\Lambda$-lattice. Define
\[
S=S_{Y} = \End_{\mathcal{M}}(\mathcal{M}Y) \quad \text{ and } \quad  
T=T_{Y} = \End_{\Lambda}(Y). \]
Note that $T$ identifies naturally with the subring $\{ f \in S \mid f(Y) \sseq Y \}$ of $S$.
Let $\mathfrak{I}$ be any full two-sided ideal of $S$ contained in $T$. 
Set $\overline{S} = S / \mathfrak{I}$ and $\overline{T} = T / \mathfrak{I}$ so that
$\overline{T}$ is a subring of $\overline{S}$, and denote the canonical map
$S \rightarrow \overline{S}$ by $s \mapsto \overline{s}$.
We have decompositions
\begin{equation}
\mathfrak{I} = \mathfrak{I}_{1} \oplus \cdots \oplus \mathfrak{I}_{r}, \quad
S = S_{1} \oplus \cdots \oplus S_{r} 
\quad  \textrm{and} \quad \overline{S} =  
\overline{S_{1}} \oplus \cdots \oplus \overline{S_{r}},
\end{equation}
where each $\mathfrak{I}_{i}$ is a full two-sided ideal of 
$S_{i} := \End_{\mathcal{M}_{i}}(\mathcal{M}_{i}Y)$
and where $\overline{S_{i}} := S_{i} / \mathfrak{I}_{i}$.
For each $i$, let $U_{i} \subseteq S_{i}^{\times} = \Aut_{\mathcal{M}_{i}}(\mathcal{M}_{i}Y)$ 
denote a set of representatives of the image of the
natural projection $S_{i}^{\times} \rightarrow (\overline{S_{i}})^{\times}$.

\begin{corollary}\label{cor:finite-check-iso}
Condition \textup{(c)} in Corollary \ref{cor:ns-breakdown} can be weakened to
\begin{itemize}
\item[(c\textprime)] there exist automorphisms $g_{i}' \in U_{i}$ such that $(\sum_{i=1}^{r} (g_{i}' \circ f_{i}))(X) \sseq Y$.  
\end{itemize}
\end{corollary}

\begin{proof}
Suppose throughout that (b) holds. If (c\textprime) holds then it is clear that (c) also holds.
Suppose conversely that (c) holds.
For each $i$, there exists $g_{i}' \in U_{i}$ such that $\overline{g_{i}'} = \overline{g_{i}}$.
Note that $g_{i}' - g_{i} \in \mathfrak{I}_{i}$.
For each $i$, define
\[
h_{i} := g_{i}' \circ g_{i}^{-1}-\mathrm{Id}_{\mathcal{M}_{i}Y}  = (g_{i}'-g_{i}) \circ g_{i}^{-1} \in \mathfrak{I}_{i}g_{i}^{-1}=\mathfrak{I}_{i}.
\]
Observe that $g_{i}' = (\mathrm{Id}_{\mathcal{M}_{i}Y} + h_{i}) \circ g_{i}$ and that
\[
h := h_{1}+\cdots+h_{r} \in \mathfrak{I}_{1} \oplus \cdots \oplus \mathfrak{I}_{r} 
= \mathfrak{I} \sseq T = \{ f \in S \mid f(Y) \sseq Y \}.
\]
Thus
\begin{eqnarray*} 
(\oplus_{i=1}^{r} (g_{i}' \circ f_{i}))(X)
&=& (\oplus_{i=1}^{r} (\mathrm{Id}_{\mathcal{M}_{i}Y} + h_{i}) \circ g_{i} \circ f_{i}))(X) \\
&=& (\mathrm{Id}_{\mathcal{M}Y} + h)\left( \oplus_{i=1}^{r} (g_{i} \circ f_{i})(X) \right) \\
&\sseq& (\mathrm{Id}_{\mathcal{M}Y} + h)(Y) \\
&\sseq& Y,
\end{eqnarray*}
that is, (c\textprime) holds. 
Therefore (c) and (c\textprime) are equivalent.
\end{proof}

\begin{remark}\label{rmk:finite-quotient-condition}
Corollary \ref{cor:finite-check-iso} is of particular interest when $\mathcal{O}$ is a residually finite Dedekind domain, that is,
a Dedekind domain such that every quotient of $\mathcal{O}$ by a non-zero ideal is finite.
For example, this condition is satisfied if $\mathcal{O}$ is the ring of integers of a number field.
Moreover, it ensures that the ring $\overline{S}$, and thus the sets of representatives $U_{i}$, are finite; 
this is crucial for the development of any algorithm that relies on Corollary \ref{cor:finite-check-iso}.
\end{remark}

\subsection{Reduction to the free rank $1$ case via homomorphism groups}\label{subsec:reduce-to-free-via-homgroups}
We now give an alternative approach to the one described in \S \ref{subsec:res-iso-max}.
Let $\mathcal{O}$ be a Dedekind domain with field of fractions $K$ and assume that $\mathcal{O} \neq K$.
Let $\Lambda$ be an $\mathcal{O}$-order in a separable $K$-algebra $A$.
Let $X$ and $Y$ be $\Lambda$-lattices. 
Then $\End_{\Lambda}(Y)$ is an $\mathcal{O}$-order in the separable $K$-algebra $\End_{A}(KY)$.
Moreover, $\Hom_{\Lambda}(X,Y)$ is a (left) $\End_{\Lambda}(Y)$-lattice in $\Hom_{A}(KX,KY)$ via post-composition.

\begin{prop}\label{prop:hom-free-over-end}
Two $\Lambda$-lattices $X$ and $Y$ are isomorphic if and only if
\renewcommand{\labelenumi}{(\alph{enumi})}
\begin{enumerate}
\item the $\End_{\Lambda}(Y)$-lattice $\Hom_{\Lambda}(X, Y)$ is free of rank $1$, and
\item every free generator of $\Hom_{\Lambda}(X, Y)$ over $\End_{\Lambda}(Y)$ is an isomorphism.
\end{enumerate}
\end{prop}

\begin{proof}
If (a) and (b) hold then it is clear that $X$ and $Y$ are isomorphic. 
Suppose conversely that $X$ and $Y$ are isomorphic.
Fix an isomorphism $\varphi \in \Hom_{\Lambda}(X, Y)$.
Then for any $g \in \Hom_{\Lambda}(X, Y)$, we have $h_{g} := g \circ \varphi^{-1} \in \End_{\Lambda}(Y)$ and so $g = h_{g} \circ \varphi$. 
Thus the map $\End_{\Lambda}(Y) \rightarrow \Hom_{\Lambda}(X, Y)$ defined by $h \mapsto h \circ \varphi$ is surjective;
since it is a map of full $\mathcal{O}$-lattices in $K$-vector spaces of equal finite dimension, it is also injective. 
Hence $\varphi$ is a free generator of  $\Hom_{\Lambda}(X, Y)$ over $\End_{\Lambda}(Y)$ and thus (a) holds.
Now let $f$ be any free generator of $\Hom_{\Lambda}(X, Y)$ over $\End_{\Lambda}(Y)$. 
Then there exists $\theta \in \Aut_{\Lambda}(Y)=\End_{\Lambda}(Y)^{\times}$ such that
$f = \theta \circ \varphi$ and hence $f$ is an isomorphism.
Thus (b) holds.
\end{proof}

\begin{remark}
Proposition \ref{prop:hom-free-over-end} is used in the isomorphism test for localized lattices given in
\S \ref{subsec:redn-to-local-freeness}.
In the case that $\mathcal{O}$ is the ring of integers of a number field $K$, necessary and sufficient conditions for the existence of a free generator of
$\Hom_{\Lambda}(X,Y)$ over $\End_{\Lambda}(Y)$
are given by (special cases of) the results of \cite[\S 2]{Bley2008}, which are themselves special cases of those in \S \ref{subsec:res-iso-max}.
\end{remark}

\section{The main algorithm}\label{sec:iso-alg}

Let $K$ be a number field with ring of integers $\mathcal{O}=\mathcal{O}_{K}$ and let $A$ be a finite-dimensional semisimple $K$-algebra.
Note that these hypotheses ensure that $A$ is a separable $K$-algebra.
Let $\Lambda$ be an $\mathcal{O}$-order in $A$.
In this section, we outline an algorithm that takes two $\Lambda$-lattices $X$ and $Y$ and 
either returns an explicit isomorphism $X \rightarrow Y$ or determines that $X$ and $Y$ are not isomorphic.
The key result on which the algorithm is based is Corollary \ref{cor:finite-check-iso} (also see Remark \ref{rmk:finite-quotient-condition}).
We require that $A$ satisfies the hypotheses (H1) and (H2) formulated in the introduction; we discuss the conditions under which these hold in \S \ref{sec:hyp}.

Before sketching the individual steps of the algorithm, we briefly describe the presentation of the input data.
We assume that $A$ is given by a $K$-basis $a_{1}, \ldots, a_{s}$ and structure constants $\alpha_{i,j,k} \in K$ for $1 \leq i,j,k \leq s$
such that $a_{i} \cdot a_{j} = \sum_{k=1}^{s} \alpha_{i,j,k} a_{k}$.
From this, it is straightforward to construct an embedding of $K$-algebras $A \rightarrow \Mat_{s \times s}(K)$
(see \cite[\S 2.2]{Eberly1989} for details).
Moreover, we assume that $\Lambda$, $X$ and $Y$ are given by $\mathcal{O}$-pseudo-bases as described, for example,
in \cite[1.4.1]{cohen-advcomp}.
In other words,
\[
X = \mathfrak{a}_{1} v_{1} \oplus \cdots \oplus \mathfrak{a}_{m} v_{m}
\quad
\textrm{ and }
\quad
Y = \mathfrak{b}_{1} w_{1} \oplus \cdots \oplus \mathfrak{b}_{n} w_{n},
\]
where for each $i$ both $\mathfrak{a}_{i}$ and $\mathfrak{b}_{i}$ are fractional ideals of $\mathcal{O}$
and $v_{i} \in V := KX$ and $w_{i} \in W := KY$.
Similarly, $\Lambda = \mathfrak{c}_{1} \lambda_{1} \oplus \cdots \oplus \mathfrak{c}_{s} \lambda_{s}$
with fractional $\mathcal{O}$-ideals $\mathfrak{c}_{i}$ and $\lambda_{i} \in A$.
To describe the action of $\Lambda$ on $X$ (and of $A$ on $V$), it suffices to assume that 
for each $i$ there is a matrix $M_{X}(\lambda_{i}) \in \mathrm{GL}_{m}(K)$ describing the action of
$\lambda_{i}$ with respect to $v_{1}, \ldots, v_{m}$. Finally, we assume that the action on $Y$ is described similarly.

\begin{algorithm}\label{alg:find-iso}
Input: $A$, $\Lambda$, $X$ and $Y$ as above.
\begin{enumerate}
\item Check that $X$ and $Y$ have equal $\mathcal{O}$-rank, that is, check that $m=n$.
\item Compute the central primitive idempotents $e_{i}$ in $A$ and the components $A_{i} :=  e_{i}A$, $1 \leq i \leq r$.
\item\label{step:comp-max-order} Compute a maximal $\mathcal{O}$-order $\mathcal{M}$ in $A$ containing $\Lambda$ and set $\mathcal{M}_{i} :=  e_{i}\mathcal{M}$.
\item Compute the set $\mathfrak{S}$ of maximal ideals $\mathfrak{p}$ of $\mathcal{O}$ dividing the module index $[\mathcal{M}:\Lambda]_{\mathcal{O}}$.
\item\label{step:max-mod-isoms} For each $i$, compute an $\mathcal{M}_{i}$-lattice isomorphism
$f_{i}\colon \mathcal{M}_{i}X \rightarrow \mathcal{M}_{i}Y$, if it exists.
\item For each $\mathfrak{p} \in \mathfrak{S}$, check that $X_{\mathfrak{p}}$ and $Y_{\mathfrak{p}}$ are isomorphic as $\Lambda_{\mathfrak{p}}$-lattices.
\item\label{step:homgroups} Compute $S = S_{Y} := \End_{\mathcal{M}}(\mathcal{M}Y)$ and 
$T = T_{Y} := \End_{\Lambda}(Y)$ as well as 
the decomposition $S = \oplus_{i=1}^{r} S_{i}$ where
$S_{i} := \End_{\mathcal{M}_{i}}(\mathcal{M}_{i}Y)$.
\item Compute a full two-sided ideal $\mathfrak{I}= \oplus_{i=1}^{r} \mathfrak{I}_{i}$ of $S$ contained in $T$.
\item For each $i$, compute a finite set of representatives 
$U_{i} \subset S_{i}^{\times}  = \Aut_{\mathcal{M}_{i}}(\mathcal{M}_{i}Y)$ of the image of the natural projection map 
$S_{i}^{\times} \rightarrow ({S_{i}}/\mathfrak I_i)^{\times}$.
\item Test whether there exists a tuple $(g_{i}) \in \prod_{i=1}^{r} U_{i}$ such that 
$(\sum_{i=1}^{r}(g_{i} \circ f_{i}))(X) \sseq Y$. \\
For such a tuple, an isomorphism is given by 
$(\sum_{i=1}^{r}(g_{i} \circ f_{i}))\colon X \rightarrow Y$.
If no such tuple exists, then $X$ and $Y$ are not isomorphic.
\end{enumerate}
\end{algorithm}

We now briefly comment on the individual steps of the algorithm. 

\begin{enumerate}
\item This is straightforward.
\item Algorithms for the computation of primitive central idempotents have been given by Eberly \cite[\S{}2.4]{Eberly1989}
and by Nebe and Steel \cite[\S 2]{MR2531228}.
\item The algorithms of Friedrichs \cite[Kapitel 3 and 4]{friedrichs} can be used to compute a maximal order $\mathcal{M}$ with
$\Lambda \subseteq \mathcal{M} \subseteq A$; it is then straightforward
to compute each $\mathcal{M}_{i}$. Alternatively, for each $i$ one can compute $\Lambda_{i} := e_{i}\Lambda$ and use the algorithm
of Nebe and Steel \cite[\S 3]{MR2531228} (or the aforementioned algorithms of Friedrichs) to construct a maximal order $\mathcal{M}_{i}$ with $\Lambda_{i} \subseteq \mathcal{M}_{i} \subseteq A_{i}$;
one can then set $\mathcal{M}:=\oplus_{i} \mathcal{M}_{i}$.
\item If the algorithms of Friedrichs are used for step \eqref{step:comp-max-order}, then $\mathfrak{S}$ can be determined along the way 
(there are additional complications if one first reduces to working in the simple components
$A_{i}$ because the containment $\Lambda \subseteq \oplus_{i} \Lambda_{i}$ is not necessarily an equality).
However, even without this observation, since both $\Lambda$ and $\mathcal{M}$ are both given by $\mathcal{O}$-pseudo-bases, $\mathfrak{S}$ can be determined using the Smith normal form algorithm
(see \cite[1.7.3]{cohen-advcomp}). 
Note that if $G$ is a finite group and $\mathcal{O}[G] \subseteq \Lambda \subseteq K[G]$ then each $\mathfrak{p} \in \mathfrak{S}$ must divide $|G|$.

\item \label{step:isom-over-max} This is the only step that requires the hypotheses (H1) and (H2).
It is described in \S \ref{sec:isom-over-max}.

\item Successful completion of step \eqref{step:max-mod-isoms} gives an isomorphism $f:=\sum_{i}f_{i} : \mathcal{M}X \rightarrow \mathcal{M}Y$ of
$\mathcal{M}$-lattices. If $\mathfrak{p}$ is a maximal ideal of $\mathcal{O}$ such that 
$\Lambda_{\mathfrak{p}} =\mathcal{M}_{\mathfrak{p}}$ then Lemma \ref{lem:ext} shows
that $f$ extends to an isomorphism 
$f_{\mathfrak{p}}: X_{\mathfrak{p}} \rightarrow Y_{\mathfrak{p}}$ of $\Lambda_{\mathfrak{p}}$-lattices.
Thus checking that $X$ and $Y$ are locally isomorphic reduces to 
checking that $X_{\mathfrak{p}}$ and $Y_{\mathfrak{p}}$ are isomorphic $\Lambda_{\mathfrak{p}}$-lattices
for each $\mathfrak{p} \in \mathfrak{S}$.
Several algorithms to check this for a given maximal ideal $\mathfrak{p}$ are described in \S \ref{sec:lociso}.
In the special case that one wants to check that both $X_{\mathfrak{p}}$ and $Y_{\mathfrak{p}}$ are in fact free 
$\Lambda_{\mathfrak{p}}$-lattices, one can use the algorithm described in \cite[\S 4.2]{Bley2009} (see \S\ref{subsec:redn-to-local-freeness}).
\item The solution of the more general problem of computing homomorphism groups is described in \S \ref{sec:saturation-homgroups}.
\item 
An algorithm of Friedrichs \cite[(2.16)]{friedrichs} can be used to compute the left conductor
$\mathfrak{c}_{l} = \{ x \in \End_{A}(KY) \mid xS \subseteq T \}$
and right conductor
$\mathfrak{c}_{r} = \{ x \in \End_{A}(KY) \mid Sx \subseteq T \}$
 of $S$ in $T$.
As noted in \cite[\S 3.2]{bley-boltje}, one can then take $\mathfrak{I}$ to be either
$\mathfrak{c}_{r} \cdot \mathfrak{c}_{l}$ or $\mathfrak{c}S$
where $\mathfrak{c} := \{ x \in Z(\End_{A}(KY)) \mid xS \subseteq T \} = \mathfrak{c}_{r} \cap C =  \mathfrak{c}_{l} \cap C $
is the central conductor of $S$ in $T$.
To minimise the running time of step \eqref{step:enumeration}, one would like to take $\mathfrak{I}$ to be as large as possible;
in many cases $\mathfrak{c}S$ is a better choice than $\mathfrak{c}_{r} \cdot \mathfrak{c}_{l}$, but in principle one could compute
both to see which is better in a given situation. 
In the case that $G$ is a finite group and $Y$ is a locally free $\mathcal{O}[G]$-module of rank $1$, 
we have canonical isomorphisms $T \cong \mathcal{O}[G]^{\op} \cong \mathcal{O}[G]$ and $S \cong \mathcal{M}^{\op} \cong  \mathcal{M}$, where ${R}^{\op}$ denotes the opposite ring of $R$; hence $\mathfrak{c}_{r}=\mathfrak{c}_{l}$
by \cite[(27.13)]{curtisandreiner_vol1}
and so we can take $\mathfrak{I}=\mathfrak{c}_{r}=\mathfrak{c}_{l}$ and use Jacobinski's conductor formula
(see loc.\ cit.\ or \cite{MR0204538}).
\item
The endomorphism ring $S_{i}= \End_{\mathcal{M}_{i}}(\mathcal{M}_{i}Y)$ is itself an $\mathcal{O}$-order in the separable $K$-algebra
$\End_{A_{i}}(e_{i}KY)$.
Viewing $S_{i}$ as a $\Z$-order, one can use the algorithm of Braun, Coulangeon, Nebe and Sch\"onnenbeck \cite{Braun2015} 
to find generators $u_{1}, \ldots, u_{t}$ of $S_{i}^{\times}$. 
One can then compute
$\overline{U}_{i}$, the subgroup of $\overline{S}_{i}^{\times}$ generated by 
$\overline{u}_{1}, \ldots, \overline{u}_{t}$ and write each element as a word of the form $\overline{u}_{1}^{r_{1}} \cdots \overline{u}_{t}^{r_{t}}$ where $r_{i} \in \Z$.
For each such word, one can take $u_{1}^{r_{1}} \cdots u_{t}^{r_{t}}$ to be the corresponding element of $U_{i}$.
Special cases of this problem are considered in \cite[\S 6]{Bley2008} and \cite[\S 7]{Bley2011}. 

\item\label{step:enumeration} The number of tests for this step can be greatly reduced by using a generalization of the methods described in \cite[\S 2]{Bley1997} and \cite[\S 7]{Bley2008}. 
This is described in \S \ref{sec:cut-number-of-tests}. Even with this improvement, this is the most time-consuming part of the whole algorithm.
\end{enumerate}

\begin{remark}\label{rmk:special-case-free}
The algorithms of Bley and the second named author of the present article given in \cite{Bley2008,Bley2011} can be viewed as
Algorithm \ref{alg:find-iso} specialised to the problem of determining whether a given $\Lambda$-lattice $X$ is in fact free and, assuming it is,
computing an explicit $\Lambda$-basis for $X$.
Thanks to the algorithm of Braun, Coulangeon, Nebe and Sch\"onnenbeck \cite{Braun2015} used in step (i) above,
the hypotheses assumed in \cite{Bley2008,Bley2011} can be weakened to those assumed in the present article (see \S \ref{sec:hyp}).
\end{remark}

\begin{remark}\label{rmk:reducing-to-free-case}
There is an alternative to Algorithm \ref{alg:find-iso} that uses the results of \S \ref{subsec:reduce-to-free-via-homgroups} to 
reduce to the `free rank $1$' case.
However, algorithmically speaking, the reduction steps are non-trivial because they require methods to 
both compute homomorphism groups (see \S \ref{sec:saturation-homgroups})
and test whether lattices are locally isomorphic (see \S \ref{sec:lociso}); indeed, these methods are needed for both approaches. 
\end{remark}

\section{Hypotheses}\label{sec:hyp}

We recall and discuss the hypotheses (H1) and (H2) required for Algorithm \ref{alg:find-iso}.
Let $K$ be a number field with ring of integers $\mathcal{O}_{K}$ and let $A$ be a finite-dimensional semisimple $K$-algebra.
Let $A=A_{1} \oplus \cdots \oplus A_{r}$ be the decomposition of $A$ into indecomposable two-sided ideals
and let $K_{i}$ denote the center of the simple algebra $A_{i}$. 

\begin{itemize}
\item[(H1)]
For each $i$, we can compute an explicit isomorphism $A_{i} \cong \Mat_{n_{i} \times n_{i}}(D_{i})$ of $K$-algebras,
where $D_{i}$ is a skew field with center $K_{i}$.
\item[(H2)] For each $i$, every maximal $\mathcal{O}_{K}$-order $\Delta_{i}$ in $D_{i}$ has the following properties:
\begin{enumerate}
\item we can solve the principal ideal problem for fractional left $\Delta_{i}$-ideals, and
\item $\Delta_{i}$ has the locally free cancellation property.
\end{enumerate}
\end{itemize}

These hypotheses are only needed for step \eqref{step:isom-over-max} of Algorithm \ref{alg:find-iso}, which is described in
detail in \S \ref{sec:isom-over-max}.
Note that (H2)(a) is independent of the choice of $\Delta_{i}$ in all cases in which it is known
(see \S \ref{subsec:principal-ideal-problem}).
Moreover, property (H2)(b) is independent of the choice of $\Delta_{i}$, that is, if it holds for some choice of $\Delta_{i}$
then it holds for all choices (see \S \ref{subsec:locally-free-cancellation}).
A detailed discussion of working without (H2)(b) is given in \S \ref{subsec:without-locally-free-cancellation}.

\subsection{Explicit isomorphisms of simple algebras - (H1)}
Note that (H1) is equivalent to explicitly finding a simple (left) $A_{i}$-module for each $i$. 
We list two situations in which this hypothesis is satisfied.
\begin{enumerate}
\item\label{step:Steels-alg}
In the case $K = \Q$, the problem in question is solved by an algorithm of Steel \cite[\S 2.3]{Steel2012}.
As described in \S \ref{sec:iso-alg}, we may assume that we have an explicit embedding of $\Q$-algebras
$A_{i} \rightarrow \Mat_{s_{i}}(\Q)$ for some $s_{i} \in \Z_{\geq 1}$. 
Then $A_{i}$ is a homogeneous module over itself, and Steel's algorithm returns simple submodules
$S_{1}, \ldots, S_{k}$ of $A_{i}$ such that $A_{i} = \oplus_{j=1}^{k} S_{j}$ and the $S_{j}$ are all isomorphic.
\item
Let $G$ be a finite group, let $K$ be a finite Galois extension of $\Q$ and let $A = K[G]$.
Then based on the character table algorithm of Unger \cite{Unger2006}, Steel \cite[\S 3.10]{Steel2012}
describes how to compute all irreducible $K[G]$-modules up to isomorphism.
\end{enumerate}

\begin{remark}
Let $K$ be a number field and let $A$ be a finite-dimensional semisimple $K$-algebra.
Then $A$ is also a finite-dimensional semisimple $\Q$-algebra.
Hence in principle (H1) is always satisfied by Steel's algorithm \cite[\S 2.3]{Steel2012} as described in \eqref{step:Steels-alg} above.
However, viewing $A$ as a $\Q$-algebra rather than a $K$-algebra means the loss of a certain amount of structural information that may slow down computations considerably.
\end{remark}

\subsection{The principal ideal problem - (H2)(a)}\label{subsec:principal-ideal-problem}
Let $D$ be a skew field that is central and finite-dimensional over a number field $F$.
(In the above notation, we will consider $D=D_{i}$ and $F=K_{i}$ for each $i$.)
Let $\Delta \subseteq D$ be a maximal $\mathcal{O}_{F}$-order and let $\mathfrak{a}, \mathfrak{b}$ be fractional left $\Delta$-ideals.
Then it is straightforward to show that $\mathfrak{a} \cong \mathfrak{b}$ as left $\Delta$-lattices if and only if 
there exists $\xi \in D^{\times}$ such that $\mathfrak{a} = \mathfrak{b} \xi$ (note that it is important that $\xi$ appears on the right side of $\mathfrak{b}$).
We say that we can solve the principal ideal problem for left ideals in $\Delta$,
if for any choice of $\mathfrak{a}$, $\mathfrak{b}$ we have an algorithm to 
\renewcommand{\labelenumi}{(\roman{enumi})}
\begin{enumerate}
\item decide whether $\mathfrak{a} \cong \mathfrak{b}$ as left $\Delta$-lattices, and
\item if so, compute $\xi \in D^{\times}$ such that $\mathfrak{a} = \mathfrak{b} \xi$.
\end{enumerate}

If $D$ is commutative (i.e.\ $D=F$) then the problem is solved by \cite[6.5.10]{cohen-comp}.
In the case that $D$ is a totally definite quaternion algebra, Demb\'el\'e and Donnelly \cite{MR2467859} described an algorithm 
and Kirschmer and Voight \cite[\S 6]{Kirschmer2010} proved that this algorithm runs in polynomial time when the base field is fixed. 
In the case that $D$ is an indefinite quaternion algebra Kirschmer and Voight \cite[\S 4]{Kirschmer2010} described an algorithm that 
improves on naive enumeration, without analysing its complexity; Page \cite{MR3240815} has given an improved algorithm and heuristic bounds for its complexity. In summary, if $D$ is either a number field or quaternion algebra, an algorithm to solve the principal ideal problem 
exists for all choices of $\Delta$ in $D$.

\subsection{Locally free cancellation - (H2)(b)}\label{subsec:locally-free-cancellation}
Let $K$ be a number field and let $\Lambda$ be an $\mathcal{O}_{K}$-order in a finite-dimensional semisimple $K$-algebra $A$.
Then $\Lambda$ is said to have the locally free cancellation property if for any locally free finitely generated left $\Lambda$-lattices $X$ and $Y$ 
we have
\[
X \oplus \Lambda^{(k)} \cong Y \oplus \Lambda^{(k)} \textrm{ for some } k \in \Z_{\geq 0} \implies X \cong Y.
\]

Let $A=A_{1} \oplus \cdots \oplus A_{r}$ be the decomposition of $A$ into indecomposable two-sided ideals
and let $K_{i}$ denote the center of the simple algebra $A_{i}$. 
We say that $A_{i}$ is Eichler/$\mathcal{O}_{K_{i}}$ if and only if $A_{i}$ is \emph{not} a totally definite quaternion algebra
(see \cite[(45.5)(i)]{curtisandreiner_vol2} or \cite[(34.4)]{Reiner2003}). More generally, $A$ is Eichler/$\mathcal{O}_{K}$
if and only if each $A_{i}$ is Eichler/$\mathcal{O}_{K_{i}}$.
The Jacobinski Cancellation Theorem \cite[(51.24)]{curtisandreiner_vol2}
says that if $A$ is Eichler/$\mathcal{O}_{K}$ then $\Lambda$ has the locally free cancellation property.

We are concerned with the situation in which $\Lambda=\Delta$ is a maximal order in a skew field $D$.
By the above discussion, we are reduced to the case that $D$ is a totally definite quaternion algebra. 
Moreover, by \cite[(51.25)]{curtisandreiner_vol2}, \cite[(17.3)(ii)]{Reiner2003} and the fact that any two orders in $D$ 
have equal completions at all but finitely many places,
we see that the locally free cancellation property is independent of the choice of $\Delta$ in $D$.
A classification of maximal orders in totally definite quaternion algebras with locally free cancellation 
is given in \cite{Hallouin2006} and corrected in \cite{Smertnig2015}. 

In principle, this classification means that one should be able to determine whether a given finite-dimensional semisimple $K$-algebra
$A$ satisfies (H2)(b), provided that one can explicitly identify the skew field $D_{i}$ in the isomorphism $A_{i} \cong \Mat_{n_{i} \times n_{i}}(D_{i})$. 
The case in which $A=K[G]$ for some finite group $G$ is of particular interest and this is discussed at length in \cite[\S 4.3]{Bley2011}.
We also have the following useful result.

\begin{lemma}
Let $K$ be a number field and let $G$ be a finite group.
Then the group algebra $K[G]$ satisfies hypothesis \textup{(H2)(b)} in each of the following cases:
\renewcommand{\labelenumi}{(\roman{enumi})}
\begin{enumerate}
\item $G$ is abelian, symmetric or dihedral;
\item $G$ is of odd order;
\item $K$ is \emph{not} totally real;
\item $K=\Q$ and $|G|<32$.
\end{enumerate}
\end{lemma}

\begin{proof}
Consider the Wedderburn decomposition $K[G] \cong \bigoplus_{i} \Mat_{n_{i} \times n_{i}}(D_{i})$ where each $D_{i}$ is a skew field with center $K_{i}$.
Then by \cite[(7.22)]{curtisandreiner_vol1} for each $i$ there exists $m_{i} \in \Z_{\geq 1}$ such that $[D_{i}:K_{i}]=m_{i}^{2}$.
In case (i), it is well known that $m_{i}=1$, i.e., that $D_{i}=K_{i}$ for each $i$.
In case (ii), each $m_{i}$ is odd as it must divide $|G|$ by \cite[(74.8)(ii)]{curtisandreiner_vol2}.
In case (iii), no $K_{i}$ is totally real as it is a field extension of $K$.
Therefore, in cases (i), (ii) and (iii), no $D_{i}$ is a totally real quaternion algebra (i.e.\ has $m_{i}=2$ and $K_{i}$ totally real), 
and so we are done by the Jacobinski Cancellation Theorem \cite[(51.24)]{curtisandreiner_vol2} as explained above.
In case (iv), the desired result follows from the proof of \cite[Lemma 4.2]{Bley2011}.
\end{proof}

\section{Isomorphisms over maximal orders in simple algebras}\label{sec:isom-over-max}

The purpose of this section is to describe step (e) of Algorithm \ref{alg:find-iso}, which is the only step that requires the hypotheses (H1) and (H2).
Thanks to hypothesis (H1), we are reduced to the following situation.
Let $D$ be a skew field that is central and finite-dimensional over a number field $F$.
(In the notation of \S \ref{sec:hyp}, we will consider $D=D_{i}$ and $F=K_{i}$ for each $i$.)
Let $\mathcal{O}=\mathcal{O}_{F}$ and fix a choice $\Delta$ of maximal $\mathcal{O}$-order in $D$.
Let $n \in \Z_{\geq 1}$ and let $\mathcal{M}$ be a maximal $\mathcal{O}$-order in $\Mat_{n \times n}(D)$.
Let $M$ and $N$ be $\mathcal{M}$-lattices.
We require an algorithm that either computes an $\mathcal{M}$-lattice isomorphism $f: M \rightarrow N$ or determines that no such isomorphism exists. 
The basic idea is to reduce this to an analogous problem for certain $\Delta$-lattices.
To this end, we first transform $\mathcal{M}$ to a maximal order of a particular shape.
We then use a `noncommutative Steinitz form' for $\Delta$-lattices
to further reduce to the principal ideal problem of (H2)(a).

\subsection{Maximal orders of a particular form}\label{subsec:max-orders-particular-form}
We first briefly recall some facts and definitions from \cite{Reiner2003}.
Let $\mathfrak{a}$ be a fractional right $\Delta$-ideal, that is, a full right $\Delta$-lattice in $D$.
The left and right orders of $\mathfrak{a}$ are defined to be
\[
\mathcal{O}_{l}(\mathfrak{a}) := \{ x\in D \mid x\mathfrak{a} \subseteq \mathfrak{a} \}
\quad
\textrm{ and }
\quad
\mathcal{O}_{r}(\mathfrak{a}) := \{ x\in D \mid \mathfrak{a}x \subseteq \mathfrak{a} \},
\]
respectively. These are $\mathcal{O}$-orders in $D$.
Since $\Delta$ is a maximal $\mathcal{O}$-order in $D$ and $\mathfrak{a}$ is a right fractional $\Delta$-ideal, 
we must have $\mathcal{O}_{r}(\mathfrak{a}) = \Delta$.
We set $\Delta' := \mathcal{O}_{l}(\mathfrak{a})$ and note that this is also a maximal $\mathcal{O}$-order in $D$ by \cite[(23.10)]{Reiner2003}.
We define
\[
\mathfrak{a}^{-1} := \{  x \in D \mid \mathfrak{a}  x  \mathfrak{a} \subseteq \mathfrak{a} \},
\]
and note that this is a fractional left $\Delta$-ideal.
Then by \cite[(22.7)]{Reiner2003} we have
\begin{equation*}\label{eq:ideal-and-inverse}
\mathfrak{a}^{-1}  \mathfrak{a} = \Delta, \quad \mathfrak{a}  \mathfrak{a}^{-1} = \Delta', \quad (\mathfrak{a}^{-1})^{-1} = \mathfrak{a}. 
\end{equation*}
Moreover, 
$\mathcal{O}_{r}(\mathfrak{a}^{-1})=\Delta'$ and $\mathcal{O}_{l}(\mathfrak{a}^{-1})=\Delta$
by \cite[(22.8)]{Reiner2003}. 
Let
\[
\mathcal{M}_{\mathfrak{a},n} := 
\begin{pmatrix} \Delta & \dotsc & \Delta & \mathfrak a^{-1} \\
\vdots & \ddots & \vdots & \vdots \\
\Delta & \dotsc & \Delta & \mathfrak a^{-1} \\
\mathfrak a & \dotsc & \mathfrak a & \Delta' \end{pmatrix}
\]
denote the ring of all $n \times n$ matrices $(x_{ij})_{1 \leq i,j\leq n}$ where $x_{11}$ ranges over all elements of $\Delta$, \dots, $x_{1n}$ ranges over all elements of $\mathfrak a^{-1}$, and so on. (In the case $n=1$, we take $\mathcal{M}_{\mathfrak a,n}=\Delta'$.)
By \cite[(27.6)]{Reiner2003} every 
maximal $\mathcal{O}$-order in $\Mat_{n \times n}(D)$ is isomorphic to $\mathcal{M}_{\mathfrak{a},n}$ for some right ideal $\mathfrak{a}$ of $\Delta$.

\begin{prop}[{{\cite[\S{}6.3]{Bley2011}}}]\label{prop-maxordspec}
There exists an algorithm, that given a maximal $\mathcal{O}$-order $\mathcal{M}$ of $\Mat_{n \times n}(D)$, 
computes a right fractional ideal $\mathfrak{a}$ of $\Delta$ and an invertible matrix $S \in \GL_{n}(D)$
such that $\mathcal{M} = S \mathcal{M}_{\mathfrak{a},n} S^{-1}$.
\end{prop}

Hence by replacing $\mathcal{M}$ by $S^{-1}\mathcal{M} S$ and an $\mathcal{M}$-lattice $M$ by $S^{-1}M$, we may assume without loss of generality
that $\mathcal{M}$ is of the form $\mathcal{M}_{\mathfrak{a},n}$ for some right fractional $\Delta$-ideal $\mathfrak{a}$.

\subsection{Reducing from $\mathcal{M}$-lattices to $\Delta$-lattices}
We assume that $n \geq 2$ and that $\mathcal{M}$ is of the form $\mathcal{M}_{\mathfrak a,n}$ for some right fractional $\Delta$-ideal $\mathfrak{a}$.
By \cite[(21.7)]{Reiner2003} $\mathcal{M}$ is Morita equivalent to $\Delta$, that is, the category of left $\mathcal{M}$-modules is equivalent to the category
of left $\Delta$-modules.
We now make this equivalence partially explicit in the case of lattices by generalizing \cite[\S 6.2]{Bley2011} and adapting parts of \cite[\S 17]{MR1653294}.

We recall the convention that all modules are assumed to be left modules unless otherwise specified.
For $1 \leq i,j \leq n$ let $e_{ij} \in \Mat_{n \times n}(D)$ be the matrix with $1$ in position $(i,j)$ and $0$ everywhere else.
Then
\[
e_{ij}e_{kl} = \left\{ \begin{array}{ll}
e_{il} & \textrm{if } j=k,      \\
0 & \textrm{otherwise}.
\end{array}
\right.
\]
Recalling that $\mathcal{O}_{l}(\mathfrak{a}^{-1})=\Delta$, 
we see that $e_{11}\mathcal{M}$ (i.e.\ the `first row' of $\mathcal{M}$) is a $\Delta$-lattice.
Thus the assignment $M \mapsto e_{11}M$ induces a functor between $\mathcal{M}$-lattices and $\Delta$-lattices, 
where the corresponding map on morphisms is given by restriction,
\[
\Hom_{\Lambda}(M, N) \longrightarrow \Hom_{\Delta}(e_{11}M, e_{11}N), \ f \longmapsto f|_{e_{11}M}.
\]
One can show that this functor yields an equivalence of categories, but for our purposes the following result suffices. 

\begin{prop}\label{prop-redtoone}
Let $M$ and $N$ be $\mathcal{M}$-lattices. 
Then $M \cong N$ as $\mathcal{M}$-lattices if and only if $e_{11}M \cong e_{11}N$ as $\Delta$-lattices.
Moreover, given an isomorphism $g \colon e_{11} M \rightarrow e_{11}N$ of $\Delta$-lattices,
we can explicitly construct an isomorphism $f : M \rightarrow N$ of $\mathcal{M}$-lattices such that $f |_{e_{11}M} = g$.
\end{prop}

\begin{proof}
We note that all elements of $D$ commute with $e_{ij}$ for all $1 \leq i,j \leq n$.
This together with the assumption that $n \geq 2$ will be used throughout.

Suppose that $f \colon M \rightarrow N$ is an isomorphism of $\mathcal{M}$-lattices.
Then it is clear that the restriction $f|_{e_{11}M} : e_{11}M \rightarrow e_{11}N$ is an isomorphism of $\Delta$-lattices.

Suppose conversely that we are given an isomorphism $g: e_{11}M \rightarrow e_{11}N$ of $\Delta$-lattices.
Let $g' : e_{11} F M \rightarrow e_{11} F N$ be the unique extension of $g$ to an isomorphism of $D$-modules,
which exists by Lemma \ref{lem:ext}. For $i=1,\ldots,n$ we define
\[
f_{i}'  :  e_{ii} F M \longrightarrow e_{ii} F N, \quad e_{ii}x \longmapsto e_{i1}g'(e_{1i}x) = e_{ii}e_{i1}g'(e_{11}e_{1i}e_{ii}x),
\]
where the last equality shows that these maps are well defined.

Suppose that $1 \leq i < n$. 
Then $e_{1i} = e_{1i}e_{ii}$ and $e_{i1}=e_{i1}e_{11}$ are both elements of $\mathcal{M}$ and so we have
\[
f_{i}'(e_{ii}M) = e_{ii}e_{i1}g'(e_{11}e_{1i}e_{ii} M) \subseteq e_{ii}e_{i1} g'(e_{11}M) =  e_{ii}e_{i1}e_{11}N \subseteq e_{ii}N.
\]
Similarly, we have
\[
e_{ii}N = e_{i1}e_{11}e_{1i}N 
\subseteq e_{i1}e_{11} N = e_{i1} g'(e_{11}M) = e_{i1} g'(e_{1i}e_{i1}M)
\subseteq e_{i1}g'(e_{1i}M) = f_{i}'(e_{ii}M).
\]
Hence $f_{i}'(e_{ii}M) = e_{ii}N$.

Now consider the case $i=n$.
Since $1 \in \Delta' = \mathfrak{a}\mathfrak{a}^{-1}$ and both
$\mathfrak{a}^{-1} e_{1n}$ and $e_{n1} \mathfrak{a}$ are contained in $\mathcal{M}$, we have
\begin{align*}
f_{n}'(e_{nn}M) 
&= e_{n1} g'(e_{1n}M) 
= e_{n1} g'(e_{11}e_{1n}M)  \\
& \subseteq e_{n1}g'(e_{11} \mathfrak{a} \mathfrak{a}^{-1} e_{1n}M) \\
& \subseteq e_{n1}g'(e_{11}\mathfrak{a}M) 
= e_{n1} e_{11} \mathfrak{a} N = e_{nn}e_{n1}\mathfrak{a}N \\
& \subseteq e_{nn} N. 
\end{align*}
Similarly, we have
\begin{align*}
e_{nn}N & = e_{n1}e_{11}e_{1n}N \\ 
& \subseteq e_{n1}e_{11} \mathfrak{a} \mathfrak{a}^{-1} e_{1n}N\\
& \subseteq e_{n1}e_{11} \mathfrak{a} N = e_{n1} g'(e_{11}\mathfrak{a}M) = e_{n1} g'(e_{1n}e_{n1} \mathfrak{a} M)\\
& \subseteq e_{n1}g'(e_{1n}M) = f_{n}'(e_{nn}M). 
\end{align*}
Hence $f_{n}'(e_{nn}M) = e_{nn}N$.

We have shown that for each $i$, the map $f_{i}'$ restricts to a well-defined surjective map $f_{i} : e_{ii} N \rightarrow e_{ii}M$.
Since $e_{11} + \dotsb + e_{nn}$ is the $n \times n$ identity matrix, we have a decomposition $M = e_{11}M \oplus \dots \oplus e_{nn} M$.
Define 
\[
f : M \longrightarrow N, \quad x \mapsto f_{1}(e_{11}x) + \cdots + f_{n}(e_{nn}x),
\]
and note that this is a homomorphism of $\mathcal{M}$-lattices by construction.
Since each $f_{i}$ is surjective and
$N = e_{11}N \oplus \cdots \oplus e_{nn} N$,
we see that $f$ is also surjective.
It remains to show that $f$ is injective.
Let $x \in M$ and suppose that $f(x) = 0$.
Since the elements $e_{ii}$ are pairwise orthogonal, this implies that $e_{i1}g'(e_{1i}x) = f_{i}(e_{ii}x) = 0$ for each $i$.
Multiplying on the left by $e_{1i}$ gives
\[
0 = e_{1i}e_{i1}g'(e_{1i}x) = e_{11}g'(e_{1i}x) = g'(e_{1i}x).
\]
Since $g'$ is injective this implies that $e_{1i}x = 0$, and multiplying on the left by $e_{i1}$ then
gives $e_{i1}e_{1i}x=e_{ii}x = 0$. Hence $x=e_{11}x + \cdots + e_{nn}x = 0$.
\end{proof}

\begin{remark}
If $\mathcal{M}=\Mat_{n \times n}(\Delta)$ then the construction given in the proof of Proposition \ref{prop-redtoone} simplifies considerably
(also see  \cite[\S 17]{MR1653294} for a discussion of the explicit Morita equivalence of $\mathcal{M}$ and $\Delta$ in this case). 
Moreover, if $\mathfrak{a}$ is a principal fractional right ideal of $\Delta$ (i.e.\ $\mathfrak{a}=\xi \Delta$ for some $\xi \in D^{\times}$)
then there exists $S \in \GL_{n}(D)$ such that $S \mathcal{M}_{\mathfrak{a},n} S^{-1} = \Mat_{n \times n}(\Delta)$ and
so we are reduced to the aforementioned situation.
In particular, if $D=F$ and $F$ has class number $1$, then every fractional right ideal of $\mathcal{O}=\Delta$ is principal
(this statement can be generalized to the case $D \neq F$ by using \cite[(49.32)]{curtisandreiner_vol2}).
\end{remark}

\subsection{Noncommutative Steinitz form for $\Delta$-lattices}\label{subsec:non-cmm-Steinitz-form}
Let $M$ be a $\Delta$-lattice. Then there exists $r \in \Z_{\geq 1}$ such that $FM \cong D^{(r)}$.
Moreover, by \cite[(27.8)]{Reiner2003} there exist elements $m_{1},\ldots,m_{r} \in FM$ and a fractional left $\Delta$-ideal $\mathfrak{b}$ such that
\[
M = \Delta m_{1} \oplus \dotsb \oplus \Delta m_{r-1} \oplus \mathfrak{b} m_{r}.
\]
Such a decomposition is known as noncommutative Steinitz form of $M$. An algorithm for computing it is described in \cite[\S 5.3, \S 5.4]{Bley2011}.
Moreover, $M$ is a locally free $\Delta$-lattice by \cite[(18.10)]{Reiner2003}.
Thus if $\Delta$ has the locally free cancellation property of (H2)(b) (described in \S \ref{subsec:locally-free-cancellation})
then $r$ and the isomorphism class of $\mathfrak{b}$ uniquely determine the isomorphism class of $M$.
We therefore have the following result.

\begin{lemma}\label{lem-isosteinitz}
Let $M$ and $N$ be two $\Delta$-lattices of equal rank $r \in \Z_{\geq 1}$.
Let $\mathfrak{b}$ and $\mathfrak{c}$ be left fractional $\Delta$-ideals such that
\[ 
M \cong \Delta^{(r-1)} \oplus \mathfrak{b}
\quad \text{and} \quad
N \cong \Delta^{(r-1)} \oplus \mathfrak{c}.
\]
If $\mathfrak{b} \cong \mathfrak{c}$ as left fractional $\Delta$-ideals, then $M$ and $N$ are isomorphic.
Moreover, if $\Delta$ has the locally free cancellation property of \textup{(H2)(b)}, then the converse is also true.
\end{lemma}

\subsection{Step (e) of Algorithm~\ref{alg:find-iso}}\label{subsec:description-of-step-e}
As input, we take a maximal order $\mathcal{M}_{i}$ in $A_{i}$ and $\mathcal{M}_{i}$-lattices $\mathcal{M}_{i}X$ and $\mathcal{M}_{i}Y$ ($i$ fixed).
We describe an algorithm that either computes an isomorphism $f : \mathcal{M}_{i}X \rightarrow \mathcal{M}_{i}Y$
of $\mathcal{M}_{i}$-lattices or determines that no such isomorphism exists. 

\begin{enumerate}
\item Using (H1) we can suppose without loss of generality that $A_{i} = \Mat_{n_{i} \times n_{i}}(D_{i})$ for some $n_{i} \in \Z_{\geq 1}$ and some skew field $D_{i}$
with center $K_{i}$. We henceforth drop the subscripts from $D_{i}$ and $n_{i}$ and write $F$ in place of $K_{i}$.
Let $\mathcal{O}=\mathcal{O}_{F}$.
\item Suppose $n=1$.
Then $\mathcal{M}_{i}=\Delta$ for some maximal $\mathcal{O}$-order $\Delta$ in $D$.
Set 
\[
M:=\mathcal{M}_{i}X=e_{11}\mathcal{M}_{i}X
\quad \text{and} \quad
N:=\mathcal{M}_{i}Y=e_{11}\mathcal{M}_{i}Y,
\]
and skip to step (iv) below.
\item Suppose $n \geq 2$. Fix any maximal $\mathcal{O}$-order $\Delta$ in $D$.
Compute $S \in \GL_{n}(D)$ as in  Proposition~\ref{prop-maxordspec}.
Set
\[
\mathcal{M}_{i}' := S^{-1}\mathcal{M}_{i} S,
\quad M:=S^{-1}\mathcal{M}_{i}X
\quad \text{and}
\quad N:=S^{-1}\mathcal{M}_{i}Y.
\]
The problem is reduced to either computing an isomorphism $M \rightarrow N$ of $\mathcal{M}_{i}'$-lattices
or determining that no such isomorphism exists.
\item \label{step:X-Y-decomps}
Use \cite[\S 5.3, \S 5.4]{Bley2011} to compute decompositions
\[
\qquad
e_{11}M = \Delta m_{1} \oplus \dotsb \oplus \Delta m_{r-1} \oplus \mathfrak{b} m_{r}
\quad \text{and} \quad
e_{11}N = \Delta n_{1} \oplus \dotsb \oplus \Delta n_{s-1} \oplus \mathfrak{c} n_{s}.
\]
If $r \neq s$ then the algorithm terminates with the conclusion that $\mathcal{M}_{i}X$ and $\mathcal{M}_{i}Y$
are not isomorphic as $\mathcal{M}_{i}$-lattices.
\item
Use (H2)(a) to check whether $\mathfrak{b} \cong \mathfrak{c}$ as fractional left $\Delta$-ideals, 
and if so, return $\xi \in D$ such that $\mathfrak{b} = \mathfrak{c} \xi$.
Otherwise, (H2)(b) shows that $e_{11}M$ and $e_{11}N$ are not isomorphic as $\Delta$-lattices
and so the algorithm terminates with the conclusion that $\mathcal{M}_{i}X$ and $\mathcal{M}_{i}Y$ are not isomorphic as $\mathcal{M}_{i}$-lattices.
(In the case $n=1$ this is true because $e_{11}M=\mathcal{M}_{i}X$ and $e_{11}N=\mathcal{M}_{i}Y$ and in the case $n \geq 2$
this follows from step (iii), Proposition \ref{prop-redtoone} and Lemma \ref{lem-isosteinitz}.)
\item
If a suitable $\xi \in D$ is found in step (v) then, together with the decompositions of $e_{11}M$ and $e_{11}N$ found in step (iv),
it can be used to compute an explicit isomorphism $e_{11}M \cong e_{11}N$ of $\Delta$-lattices.
\item
If $n=1$, then we are already done since $\mathcal{M}_{i}=\Delta$, $e_{11}M=\mathcal{M}_{i}X$ and $e_{11}N=\mathcal{M}_{i}Y$. 
If $n \geq 2$ then use  Proposition~\ref{prop-redtoone} to construct an isomorphism $M \rightarrow N$ of $\mathcal{M}_{i}'$-lattices;
using step (iii), we thus obtain an isomorphism $\mathcal{M}_{i}X \rightarrow \mathcal{M}_{i}Y$ of $\mathcal{M}_{i}$-lattices.
\end{enumerate}

\begin{remark}
If $n=1$ then $\Delta$ is uniquely determined by $\mathcal{M}_{i}$.
Moreover, the choice of $\mathcal{M}_{i}$ may be limited by the requirement that $\Lambda \subseteq \oplus_{i} \mathcal{M}_{i}$.
However, if $n \geq 2$ then we can make any choice of $\Delta$.
\end{remark}

\subsection{Working without the locally free cancellation property (H2)(b)}\label{subsec:without-locally-free-cancellation}

\begin{remark}\label{rmk:r=1}
Suppose that $KM \cong KN \cong D \cong A_{i}$, that is, $n=r=s=1$ and $\mathcal{M}_{i}=\Delta$ in the notation above.
Then the problem of determining whether $e_{11}N \cong e_{11}M$ as $\Delta$-lattices
is equivalent to checking whether $\mathfrak{b} \cong \mathfrak{c}$ as fractional left $\Delta$-ideals.
This can be done with (H2)(a) alone and so (H2)(b) is not necessary in this case.
Thus (H2)(b) can be replaced with a weaker but more complicated to state hypothesis,
which corresponds to (H2\textprime)(a) of \cite{Bley2011}.
\end{remark}

\begin{example}\label{ex:Q32-H2b}
Let $Q_{32}$ be the quaternion group of order $32$ and let $A=\Q[Q_{32}]$.
Let $\Lambda$ be any order in $A$ and let $\mathcal{M}$ be a maximal order such that $\Lambda \subseteq \mathcal{M} \subseteq A$.
Then \cite[Theorem II]{Swan1983} shows that $\mathcal{M}$ does not have the locally free cancellation property.
Each simple component $A_{i}$ of $A$ is isomorphic to either 
(i) a matrix ring over a number field or (ii) a totally definite quaternion algebra.
By the discussion in \S \ref{subsec:locally-free-cancellation}, each $\mathcal{M}_{i}$ in a component $A_{i}$ of case (i) does have the locally free cancellation property. The remaining $\mathcal{M}_{i}$ in components $A_{i}$ of case (ii) do not have the locally free cancellation property.
However, if $X$ and $Y$ are $\Lambda$-lattices such that $\Q X \cong \Q Y \cong A$ then Remark \ref{rmk:r=1}
shows that Algorithm~\ref{alg:find-iso} will always correctly determine whether $\mathcal{M}X$ and $\mathcal{M}Y$ are isomorphic
as $\mathcal{M}$-lattices in step (e); the other steps will run as usual as they do not depend on (H2).
\end{example}

\begin{remark}\label{rmk:run-without-locally-free-cancellation}
The algorithm described in \S \ref{subsec:description-of-step-e} can still be run without (H2)(b), but it may not come to a conclusion.
If it is found that $\mathfrak{b} \cong \mathfrak{c}$ as fractional left $\Delta$-ideals in step (v), 
then the algorithm will go on to construct an isomorphism $\mathcal{M}_{i}X \rightarrow \mathcal{M}_{i}Y$ of $\mathcal{M}_{i}$-lattices,
whether or not  $\Delta$ has the locally free cancellation property.
However, if $\mathfrak{b} \not \cong \mathfrak{c}$ and $\Delta$ does not have the locally free cancellation property, then it is not possible to 
conclude that $e_{11}M \not \cong e_{11}N$ and thus the algorithm cannot determine whether $\mathcal{M}_{i}X$ and $\mathcal{M}_{i}Y$ are
isomorphic as $\mathcal{M}_{i}$-lattices.
\end{remark}

\section{Saturation of lattices and computation of homomorphism groups}\label{sec:saturation-homgroups}

We describe how to compute homomorphism groups between lattices over orders and 
along the way give algorithms to compute saturations of lattices.
Thus, in particular, we can compute endomorphism rings of lattices over orders, as required for step (g) of Algorithm \ref{alg:find-iso}. 
Some of the results presented here were already given in the first named author's Ph.D.\ thesis \cite[\S 1, \S 11]{Hofmann2016PhD}.

\subsection{Saturations of lattices over integral domains}\label{subsec:saturations-of-lattices}
Let $R$ be an integral domain with field of fractions $K$. To avoid trivialities, we assume that $R \neq K$.
Let $M$ be an $R$-lattice, that is, a finitely generated torsion-free $R$-module.
Let $N$ be an $R$-sublattice of $M$.
We say that $N$ is saturated (in $M$) if $M/N$ is torsion-free as an $R$-module
(note that the term `$R$-pure sublattice' is used in \cite[p.\ 45]{Reiner2003}).
Equivalently, $N$ is saturated in $M$ if for each $r \in R$ we have $N \cap rM = rN$.
The unique minimal $R$-lattice $L$ that is saturated in $M$ and satisfies $N \subseteq L \subseteq M$ is called the saturation of $N$ in $M$.

\begin{lemma}\label{lem:saturation}
The saturation of $N$ in $M$ is $KN \cap M$.
\end{lemma}

\begin{proof}
This follows directly from \cite[(4.0)]{Reiner2003};
also see \cite[Lemma 1.33(ii)]{Hofmann2016PhD}.
\end{proof}

\subsection{Computation of saturations of lattices over rings of integers}\label{subsec:comp-sat}

Let $K$ be a number field with ring of integers $\mathcal{O}=\mathcal{O}_{K}$.
We now give a result that can be used to compute saturations of $\mathcal{O}$-lattices.
A key ingredient is the pseudo-Hermite normal form as described in \cite[1.4.6]{cohen-advcomp}.
A similar result that uses the pseudo-Smith normal form \cite[1.7.2]{cohen-advcomp} instead
is given in the first named author's Ph.D.\ thesis \cite[Lemma 1.35]{Hofmann2016PhD}.
In the special case $\mathcal{O}=\Z$, both results are contained in Steel's Ph.D.\ thesis \cite[1.7.9]{Steel2012}.
While the result below is folklore, to the best of the authors' knowledge no proof has been published until now.

\begin{lemma}\label{lem:satalg2}
Let $M$ and $N$ be $\mathcal{O}$-lattices with pseudo-bases
$((\alpha_{i})_{1\leq i \leq m}, (\mathfrak{a}_{i})_{1 \leq i \leq m})$ and 
$((\beta_{i})_{1 \leq i \leq n}, (\mathfrak{b}_{i})_{1 \leq i \leq n})$, respectively.
Suppose that $N$ is a sublattice of $M$ (so $n \leq m$).
Let $A \in \operatorname{Mat}_{m \times n}(K)$ be the unique matrix satisfying
\[
(\beta_{1}, \beta_{2}, \dotsc, \beta_{n}) = (\alpha_{1}, \alpha_{2}, \dotsc,\alpha_{m}) A.
\]
Let $((\mathfrak{h}_{i})_{1\leq i \leq m}, H)$ be a pseudo-Hermite normal form of $((\mathfrak{a}_{i}^{-1})_{1 \leq i \leq m}, A^t)$ and let $U \in \GL_{m}(K)$
be the corresponding transformation matrix.
Define
\[
  (\omega_{1}, \omega_{2}, \dotsc, \omega_{m}) := (\alpha_{1}, \alpha_{2}, \dotsc,\alpha_{m})U^{-t}.
\]
Then $((\mathfrak h_{i}^{-1})_{m -n + 1 \leq i \leq m}, (\omega_{i})_{m - n + 1 \leq i \leq m})$ is a pseudo-basis of the saturation of $N$ in $M$.
\end{lemma}

\begin{proof}
For a matrix $P$ we will denote the entry at position $(i,j)$ with $P_{i,j}$.
By definition of pseudo-Hermite normal form (see \cite[1.4.6]{cohen-advcomp}) the following hold:
\renewcommand{\labelenumi}{(\alph{enumi})}
\begin{enumerate}
  \item For all $i$ and $j$ we have $U_{i,j} \in \mathfrak{a}_{i}^{-1}\mathfrak{h}_{j}^{-1}$.
  \item We have $\prod_{i=1}^{m} \mathfrak{a}_{i}^{-1} = \det(U) \prod_{i=1}^{m} \mathfrak{h}_{i}$.
\item The matrix $H=A^t U$ is of the following form
\[
H=A^{t} U= 
\left( 
\begin{array}{cccccccc}
0 & 0 & \cdots & 0 & 1 & * & \cdots &  * \\
0 & 0 & \cdots & 0 & 0 & 1 & \cdots & * \\
\vdots & \vdots & \ddots & \vdots & \vdots & \ddots & \ddots & \vdots \\
0 & 0 & \cdots & 0 & 0 & \cdots & 0 & 1
\end{array}
\right),
\]
where the first $m-n$ columns are zero (we will write this in abbreviated
form as $H = A^{t} U=(0 \, | \, {\tilde H})$, where $\tilde H \in \Mat_{n \times n}(K)$).
\end{enumerate}

We first show that $((\mathfrak{h}_{i}^{-1})_{1 \leq i \leq m}, (\omega_{i})_{1 \leq i \leq m})$ is a pseudo-basis of $M$.
By \cite[1.4.2]{cohen-advcomp} it suffices to show that
\[
\textstyle{
\prod_{i=1}^m \mathfrak{a}_{i} = \det(U^{-t}) \prod_{i=1}^{m} \mathfrak{h}_{i}^{-1}
\quad \textrm{ and } \quad 
(U^{-t})_{i,j} \in \mathfrak{a}_{i} (\mathfrak{h}_{j}^{-1})^{-1} =  \mathfrak{a}_{i} \mathfrak{h}_{j} \textrm{ for } 1 \leq i,j \leq m.}
\]
The first claim follows directly from property (b).
The second claim is equivalent to $(U^{-1})_{i,j} \in \mathfrak{a}_{j} \mathfrak{h}_{i}$ for $1 \leq i,j \leq m$,
which can be shown by expressing $(U^{-1})_{i, j}$ in terms of the adjugate matrix and using properties (a) and (b).

Let $S = \bigoplus_{i=m-n+1}^{m} \mathfrak{h}_{i}^{-1} \omega_{i}$.
Then $M/S \cong \bigoplus_{i=1}^{m - n} \mathfrak{h}_{i}^{-1} \omega_{i}$ is torsion-free
and so $S$ is saturated in $M$.
Moreover, since $S$ and $N$ are both of rank $n$, to prove that $S$ is the saturation of $N$ in $M$ it suffices to show that $N \subseteq S$.

Observe that
\begin{align*}
  (\beta_{1}, \dotsc, \beta_{n})
  &= (\alpha_{1}, \dotsc, \alpha_{m})A = (\alpha_{1}, \dotsc, \alpha_{m}) U^{-t}H^{t}\\
  &= (\omega_{1}, \dotsc, \omega_{m})H^{t} = (\omega_{1}, \dotsc, \omega_{m}) \begin{pmatrix} 0 \\ \tilde{H}^{t} \end{pmatrix} \\
  &= (\omega_{m - n + 1}, \dotsc, \omega_{m}) \tilde{H}^{t}.
\end{align*}
By our hypotheses, we have
$
\mathfrak{b}_{i} \sum_{j=1}^{m} \alpha_{j} A_{j,i} = \mathfrak{b}_{i} \beta_{i} \subseteq N \subseteq M = \oplus_{j=1}^{m} \mathfrak{a}_{j}\alpha_{j}
$
for $1 \leq i \leq n$.
Hence for $1 \leq i \leq n$ and $1 \leq j \leq m$ we have $\mathfrak{b}_{i} A_{j,i} \subseteq \mathfrak{a}_{j}$,
that is, $A_{j,i} \in \mathfrak{a}_{j}\mathfrak{b}_{i}^{-1}$.
Together with property (a), this shows that for $1 \leq i, j \leq n$, we have
 \begin{align*}
  \tilde{H}_{i, j} = H_{i, m - n + j}
   &= \sum_{k=1}^{m} (A^{t})_{i, k} U_{k, m - n + j}
   = \sum_{k=1}^{m} A_{k, i} U_{k, m - n + j}\\
   &\in \sum_{k=1}^{m} \mathfrak{a}_{k} \mathfrak{b}_{i}^{-1} \mathfrak{a}_{k}^{-1} \mathfrak{h}_{m - n + j}^{-1}
   = \mathfrak{b}_{i}^{-1}\mathfrak{h}_{m - n + j}^{-1}. 
\end{align*}
In particular, for $1 \leq j \leq n$ we have
\[
\mathfrak{b}_{j} \beta_{j}
= \mathfrak{b}_{j} \left( \sum_{i=1}^{n} (\tilde{H}^{t})_{i, j} \cdot \omega_{m - n + i}\right)
\subseteq \mathfrak{b}_{j} \sum_{i=1}^{n} \mathfrak{b}_{j}^{-1} \mathfrak{h}_{m - n + i}^{-1} \omega_{m - n + i}
= \sum_{i=1}^{n} \mathfrak{h}_{m - n + i}^{-1} \omega_{m - n + i} = S.
\]
This shows that $N = \sum_{j=1}^{n} \mathfrak{b}_{j} \beta_{j} \subseteq S$, as required.
\end{proof}

\begin{remark}\label{rmk:algs-HNF-SNF}
An algorithm to compute pseudo-Hermite normal forms is given by \cite[Algorithm 1.4.7]{cohen-advcomp}
and has been improved upon in \cite{Fieker2014, Biasse2016}.
\end{remark}

\subsection{Computation of homomorphism groups}\label{subsec:comp-homgroups}
Let $K$ be a number field with ring of integers $\mathcal{O}=\mathcal{O}_{K}$ and let $A$ be a finite-dimensional semisimple $K$-algebra.
Let $\Lambda$ be an $\mathcal{O}$-order in $A$.
We give an algorithm that takes two $\Lambda$-lattices $X$ and $Y$ and returns the $\Lambda$-lattice $\Hom_{\Lambda}(X, Y)$.
Let $V=KX$ and $W=KY$, which may be regarded as $A$-modules. 
A key ingredient in computing the homomorphism group is the following characterization,
\begin{align}\label{eq:homring} 
\Hom_\Lambda(X, Y) = \{ f|_{X} \mid f \in \Hom_{A}(V, W) \text{ such that $f(X) \subseteq Y$} \}, 
\end{align}
where $f|_{X}$ denotes the restriction of $f$ to a map $f \colon X \to Y$.
Note that this follows from the fact that every element in $\Hom_{\Lambda}(X, Y)$ extends uniquely to an element in $\Hom_{A}(V, W)$ (see Lemma~\ref{lem:ext}).
Since there exists an algorithm for computing a $K$-basis of $\Hom_{A}(V, W)$ due to Steel \cite[1.9.3]{Steel2012},
it remains to single out the morphisms that map $X$ to $Y$. 
This will be done by employing the algorithm for computing saturations given in \S \ref{subsec:comp-sat}.
We assume that $X$ and $Y$ are given by pseudo-bases, that is,
\[
X = \mathfrak{a}_{1} \alpha_{1} \oplus \mathfrak{a}_{2} \alpha_{2} \oplus \dotsb \oplus \mathfrak{a}_{m} \alpha_{m} \quad \text{ and } \quad
Y = \mathfrak{b}_{1} \beta_{1} \oplus \mathfrak{b}_{2} \beta_{2} \oplus \dotsb \oplus \mathfrak{b}_{n} \beta_{n},
\]
with $\alpha_{i} \in V, \beta_{j} \in W$ and $\mathfrak{a}_{i}, \mathfrak{b}_{j}$ fractional ideals of $K$.
Since $(\alpha_{i})_{1 \leq i \leq m}$ and $(\beta_{j})_{1 \leq j \leq n}$ are
$K$-bases of $V$ and $W$ respectively, we use them to identify $\Hom_{K}(V, W)$ with $\Mat_{m \times n}(K)$
(we treat vectors as row vectors).
Under this identification we have
\begin{equation}\label{eq:Hom-O-containment}
\Hom_{\mathcal{O}}(X, Y)
= \bigoplus_{\substack{1 \leq i \leq n \\ 1 \leq j \leq m}} \mathfrak{a}_{i}^{-1}\mathfrak{b}_{j} e_{ij}
= \begin{pmatrix}
\mathfrak{a}_{1}^{-1}\mathfrak{b}_{1} & \dotsc & \mathfrak{a}_{1}^{-1}\mathfrak{b}_{n} \\
\vdots & \ddots & \vdots \\
\mathfrak{a}_{m}^{-1}\mathfrak{b}_{1} & \dotsc & \mathfrak{a}_{m}^{-1}\mathfrak{b}_{n} \end{pmatrix} \subseteq \Mat_{m \times n}(K), 
\end{equation}
where $e_{ij} \in \Mat_{m \times n}(K)$ is the matrix with $1$ in position $(i,j)$ and $0$ everywhere else.
The following result is a special case of \cite[Lemma 11.2]{Hofmann2016PhD}. 

\begin{lemma}\label{lem:homsat}
\renewcommand{\labelenumi}{(\alph{enumi})}
\begin{enumerate}
\item 
We have $\Hom_{\Lambda}(X, Y) = \Hom_{A}(V, W) \cap \Hom_{\mathcal{O}}(X, Y)$.
\item
Let $B_{1},\dotsc,B_{r}$ be a $K$-basis of $\Hom_{A}(V, W)$ with $B_{i} \in \Hom_{\mathcal{O}}(X, Y)$.
Then
\[
\Hom_{\Lambda}(X, Y) = K(\langle B_{1}, \dotsc, B_{r} \rangle_{\mathcal{O}}) \cap \Hom_{\mathcal{O}}(X, Y).
\]
\end{enumerate}
\end{lemma}

\begin{proof}
Since a map $f \in \Hom_{A}(V, W)$ is also $\mathcal{O}$-linear we have $f(X) \subseteq Y$ if and only if $f \in \Hom_{\mathcal{O}}(X, Y)$.
Thus part (a) follows from \eqref{eq:homring}. Part (b) follows from (a) since 
$K(\langle B_{1}, \dotsc, B_{r} \rangle_{\mathcal{O}}) = \Hom_{A}(V, W)$.
\end{proof}

Using Lemma \ref{lem:saturation} we see that Lemma \ref{lem:homsat}(b) says that 
the $\mathcal{O}$-lattice $\Hom_{\Lambda}(X, Y)$ is the saturation of 
$\langle B_{1}, \dotsc, B_{r} \rangle_{\mathcal{O}}$ in $\Hom_{\mathcal{O}}(X, Y)$.
Hence the computation of $\Hom_{\Lambda}(X, Y)$ is reduced to the computation of a saturation of $\mathcal{O}$-lattices.
Therefore to compute $\Hom_{\Lambda}(X, Y)$, we proceed as follows:
\begin{enumerate}
\item Compute a $K$-basis $B_{1},\dotsc,B_{r}$ of $\Hom_{A}(V, W)$ using the algorithm of Steel \cite[1.9.3]{Steel2012}. 
\item
Scale the $B_{i}$'s such that $B_{i} \in \Hom_{\mathcal{O}}(X, Y)$ for all $i = 1,\dotsc,r$.
\item Compute the saturation $S$ of $\langle B_{1},\dotsc,B_{r} \rangle_\mathcal{O}$ in
$\Hom_{\mathcal{O}}(X,Y)$ using Lemma \ref{lem:satalg2} and one of the algorithms referenced in Remark \ref{rmk:algs-HNF-SNF}, 
and return $S$.
\end{enumerate}

\section{Isomorphism testing for localized lattices}\label{sec:lociso}

Let $K$ be a number field with ring of integers $\mathcal{O}=\mathcal{O}_{K}$ and let $A$ be a finite-dimensional semisimple $K$-algebra.
Let $\Lambda$ be an $\mathcal{O}$-order in $A$ and let $\mathcal{M}$ be a maximal $\mathcal{O}$-order such that
$\Lambda \subseteq \mathcal{M} \subseteq A$.
Let $X$ and $Y$ be $\Lambda$-lattices of equal $\mathcal{O}$-rank $n$
and let $\mathfrak{p}$ be a maximal ideal of $\mathcal{O}$.
We give four algorithms that determine whether or not the localizations
$X_{\mathfrak{p}}$ and $Y_{\mathfrak{p}}$ are isomorphic as $\Lambda_{\mathfrak{p}}$-lattices.
The first algorithm makes use of reduced lattices,
the second uses the results of \S \ref{sec:saturation-homgroups}, 
the third is a probabilistic algorithm of Monte Carlo type
and the fourth reduces to the case of testing whether a lattice localized at $\mathfrak{p}$ is free.
Of these, only the second and the fourth algorithms explicitly compute an isomorphism, if one exists. 
Variants of the first three algorithms are contained in the first named author's Ph.D.\ thesis \cite[\S 12]{Hofmann2016PhD},
but we caution that here the subscript $\mathfrak{p}$ denotes localization, whereas in loc.\ cit.\ it denotes completion.
Finally, we discuss implementations and running times of the last three algorithms in the case that $G$ is a finite group and
$\Lambda=\Z[G]$.

\subsection{Using reduced lattices}\label{subsec:redlat}

\begin{prop}\label{prop:higman}
Let $k_{0} := \min \{ k \in \Z_{\geq 0} \mid \mathfrak{p}^{k} \mathcal{M}_{\mathfrak{p}} \subseteq \Lambda_{\mathfrak{p}} \}$
and let $k \geq k_{0}+1$. Then the following are equivalent:
\renewcommand{\labelenumi}{(\alph{enumi})}
\begin{enumerate}
\item $X_{\mathfrak{p}}$ and $Y_{\mathfrak{p}}$ are isomorphic as $\Lambda_{\mathfrak{p}}$-lattices.
\item $X/\mathfrak{p}^{k} X$ and $Y/\mathfrak{p}^{k} Y$ are isomorphic as $\Lambda/\mathfrak{p}^{k} \Lambda$-modules.
\item The $\mathcal{O}/\mathfrak{p}^{k}$-module $\Hom_{\Lambda/\mathfrak{p}^{k} \Lambda}(X/\mathfrak{p}^{k} X, Y/\mathfrak{p}^{k} Y)$
contains an invertible element. 
\end{enumerate}
\end{prop}

\begin{proof}
The equivalence of (b) and (c) is clear.
Since $\mathfrak{p}^{k_{0}} \mathcal{M}_{\mathfrak{p}} \subseteq \Lambda_{\mathfrak{p}}$ we see that 
$\mathfrak{p}^{k_{0}} \mathcal{O}_{\mathfrak{p}}$ is contained in the central conductor of
$\Lambda_{\mathfrak{p}}$ in $\mathcal{M}_{\mathfrak{p}}$.
From \cite[(29.4)]{curtisandreiner_vol1} it follows that $\mathfrak{p}^{k_{0}} \mathcal
O_{\mathfrak{p}} \cdot \Ext^{1}_{\Lambda_{\mathfrak{p}}}(M, N) = 0$ for all
$\Lambda_{\mathfrak{p}}$-lattices $M$ and $N$.
Now a theorem of Higman \cite[Theorem 3]{MR0109175} (also see \cite[(30.14)]{curtisandreiner_vol1}) implies that
$X_{\mathfrak{p}}$ and $Y_{\mathfrak{p}}$ are isomorphic $\Lambda_{\mathfrak{p}}$-lattices if and only if 
$X_{\mathfrak{p}}/\mathfrak{p}^{k}X_{\mathfrak{p}}$ and $Y_{\mathfrak{p}}/\mathfrak{p}^{k} Y_{\mathfrak{p}}$ are isomorphic as
$\Lambda_{\mathfrak{p}}/\mathfrak{p}^{k} \Lambda_{\mathfrak{p}}$-modules.
The equivalence of (a) and (b) now follows from the canonical isomorphisms  
\[
\Lambda_{\mathfrak{p}}/\mathfrak{p}^{k} \Lambda_{\mathfrak{p}} \cong \Lambda/\mathfrak{p}^{k} \Lambda,
\quad
X/\mathfrak{p}^{k} X \cong X_{\mathfrak{p}}/\mathfrak{p}^{k} X_{\mathfrak{p}}
\quad 
\textrm{ and }
\quad
Y/\mathfrak{p}^{k} Y \cong Y_{\mathfrak{p}}/\mathfrak{p}^{k} Y_{\mathfrak{p}}.
\qedhere
\]
\end{proof}

To exploit this result algorithmically, we first have to explain how to determine the homomorphism group of reduced modules in part (c).
Note that $X/\mathfrak{p}^{k} X$ and $Y/\mathfrak{p}^{k} Y$ are both free of rank $n$ over $\mathcal{O}/\mathfrak{p}^k$.
Thus we may fix $\mathcal{O}/\mathfrak{p}^{k}$-bases of $X/\mathfrak{p}^{k} X$ and $Y/\mathfrak{p}^{k} Y$, which we use to
describe the action of $\Lambda$ on these modules via ring homomorphisms
\begin{align*}
& \rho_1 \colon \Lambda \longrightarrow \End_{\Lambda/\mathfrak{p}^k\Lambda}(X/\mathfrak{p}^k X) \longrightarrow \Mat_{n \times n}(\mathcal{O}/\mathfrak{p}^k) \\
& \rho_2 \colon \Lambda \longrightarrow \End_{\Lambda/\mathfrak{p}^k\Lambda}(Y/\mathfrak{p}^k Y) \longrightarrow \Mat_{n \times n}(\mathcal{O}/\mathfrak{p}^k).  
\end{align*}
Then for a set $\mathcal G \subseteq \Lambda$ generating $\Lambda$ as an $\mathcal{O}$-algebra we obtain
\[
\Hom_{\Lambda/\mathfrak{p}^{k} \Lambda}(X/\mathfrak{p}^{k} X, Y/\mathfrak{p}^{k} Y)
= \{ M \in \Mat_{n \times n}(\mathcal{O}/\mathfrak{p}^{k}) \mid \rho_1(g)M = M \rho_2(g) \, \forall g \in \mathcal{G} \}.
\]
Consider the $\mathcal{O}/\mathfrak{p}^{k}$-linear map
\begin{equation}\label{eq:h-map}
h \colon \Mat_{n \times n}(\mathcal{O}/\mathfrak{p}^{k}) \to \prod_{g \in \mathcal{G}} \Mat_{n \times n}(\mathcal{O}/\mathfrak{p}^{k}), 
\quad 
M \mapsto ((\rho_{1}(g) M - M \rho_{2}(g))_{g \in \mathcal{G}}), 
\end{equation}
and observe that $\ker(h) = \Hom_{\Lambda/\mathfrak{p}^{k} \Lambda}(X/\mathfrak{p}^{k} X, Y/\mathfrak{p}^{k} Y)$.
Since the quotient ring $\mathcal{O}/\mathfrak{p}^{k}$ is an Euclidean ring in the sense of~\cite{Fletcher1971},
an $\mathcal{O}/\mathfrak{p}^{k}$-spanning set of $\ker(h)$
can be computed using techniques related to the Howell normal form (see \cite{Storjohann1998, Fieker2014}).

Let $M \mapsto \overline{M}$ denote the canonical projection map $\Mat_{n \times n}(\mathcal{O}/\mathfrak{p}^{k}) \to \Mat_{n \times n}(\mathcal{O}/\mathfrak{p})$.

\begin{lemma}
Let $k$ be as in Proposition~\ref{prop:higman} and let $A_{1},\dotsc,A_{r} \in \Mat_{n \times n}(\mathcal{O}/\mathfrak{p}^{k})$ be an
$\mathcal{O}/\mathfrak{p}^{k}$-spanning set of $\Hom_{\Lambda/\mathfrak{p}^{k} \Lambda}(X/\mathfrak{p}^{k} X, Y/\mathfrak{p}^{k} Y)$.
Then $X_{\mathfrak{p}}$ and $Y_{\mathfrak{p}}$ are isomorphic as $\Lambda_{\mathfrak{p}}$-lattices
if and only if there exist $a_{1},\dotsc,a_{r} \in \mathcal{O}/\mathfrak{p}$ with
$\det_{\mathcal{O}/\mathfrak{p}}(a_{1} \overline{A}_{1} + \dotsb + a_{r} \overline{A}_{r}) \neq 0$.
\end{lemma}

\begin{proof}
By Proposition~\ref{prop:higman}, $X_{\mathfrak{p}}$ and $Y_{\mathfrak{p}}$ are isomorphic as $\Lambda_{\mathfrak{p}}$-lattices
if and only if there exist elements $a_{1},\dotsc,a_{r} \in \mathcal{O}/\mathfrak{p}^{k}$ such that $a_{1} A_{1} + \dotsb + a_{r} A_{r}$ is invertible,
or equivalently, $\det_{\mathcal{O}/\mathfrak{p}^{k}}(a_{1} A_{1} + \dotsb + a_{r} A_{r}) \neq 0$.
The claim now follows since reduction mod $\mathfrak{p}$ commutes with taking determinants and an element $a \in \mathcal{O}/\mathfrak{p}^{k}$ is a unit if and only if
$(a \bmod{\mathfrak{p}})$ is a unit in $\mathcal{O}/\mathfrak{p}$.
\end{proof}

To test whether $X_{\mathfrak{p}}$ and $Y_{\mathfrak{p}}$ are isomorphic as $\Lambda_{\mathfrak{p}}$-lattices,
we can thus proceed as follows:
\begin{enumerate}
\item Use an algorithm of Friedrichs \cite[(2.16)]{friedrichs} to compute $k_{0}$.
Set $k:=k_{0}+1$.
\item
Construct the $\mathcal{O}/\mathfrak{p}^{k}$-linear map $h$ of \eqref{eq:h-map} as a matrix.
\item
Compute an $\mathcal{O}/\mathfrak{p}^{k}$-spanning set $A_{1},\dotsc,A_{r}$ of $\ker(h)$.
\item
For every tuple $(a_{1},\dotsc,a_{r}) \in (\mathcal{O}/\mathfrak{p})^{r}$ test whether
$\det_{\mathcal{O}/\mathfrak{p}}(a_{1} \overline{A}_{1} + \dotsb + a_{r} \overline{A}_{r}) \neq 0$. 
The $\Lambda_{\mathfrak{p}}$-lattices $X_{\mathfrak{p}}$ and $Y_{\mathfrak{p}}$ are isomorphic if and only if such a tuple exists.
\end{enumerate}

\subsection{Using global homomorphism groups}\label{subsec:globhom}

The second approach is based on the ability to compute the global homomorphism group $\Hom_{\Lambda}(X, Y)$.
We follow the notation and setup of \S \ref{subsec:comp-homgroups}, but specialize to the case $m=n$.
Hence we let $V=KX$ and $W=KY$ and assume that $X$ and $Y$ are given by pseudo-bases, that is,
\[
X = \mathfrak{a}_{1} \alpha_{1} \oplus \mathfrak{a}_{2} \alpha_{2} \oplus \dotsb \oplus \mathfrak{a}_{n} \alpha_{n} \quad \text{ and } \quad
Y = \mathfrak{b}_{1} \beta_{1} \oplus \mathfrak{b}_{2} \beta_{2} \oplus \dotsb \oplus \mathfrak{b}_{n} \beta_{n}, 
\]
with $\alpha_{i} \in V, \beta_{j} \in W$ and $\mathfrak{a}_{i}, \mathfrak{b}_{j}$ fractional ideals of $K$.
As in \S \ref{subsec:comp-homgroups}, we use the $K$-bases $(\alpha_{i})_{i}$, $(\beta_{i})_{i}$ to identify homomorphism spaces as subsets of
$\Mat_{n \times n}(K)$. Thus we have
\begin{equation}\label{eq:identify-homgroups}
\Hom_{\Lambda}(X, Y) \subseteq \Hom_{A}(V, W) \subseteq \Hom_{K}(V, W) = \Mat_{n \times n}(K). 
\end{equation}
Let $M \mapsto \overline{M}$ denote the canonical projection map $\Mat_{n \times n}(\mathcal{O}_{\mathfrak{p}}) \to \Mat_{n \times n}(\mathcal{O}/\mathfrak{p})$.

\begin{lemma}\label{lem:padichom}
Suppose $v_{\mathfrak{p}}(\mathfrak{a}_{i}) = v_{\mathfrak{p}}(\mathfrak{b}_{i}) = 0$ for $1 \leq i \leq n$. 
Then the following hold:
\renewcommand{\labelenumi}{(\alph{enumi})}
\begin{enumerate}
\item
The $\Lambda_{\mathfrak{p}}$-lattices $X_{\mathfrak{p}}$ and $Y_{\mathfrak{p}}$ are $\mathcal{O}_{\mathfrak{p}}$-free with
$\mathcal{O}_{\mathfrak{p}}$-bases $(\alpha_{i})_{i}$ and $(\beta_{j})_{j}$, respectively.
\item We have $\Hom_{\Lambda_{\mathfrak{p}}}(X_{\mathfrak{p}}, Y_{\mathfrak{p}}) = \Mat_{n \times n}(\mathcal{O}_{\mathfrak{p}}) \cap \Hom_{A}(V, W)$.
\item
There exists a pseudo-basis $((\mathfrak{c}_{i})_{i}, (A_{i})_{i})_{1 \leq i \leq r}$ of
$\Hom_{\Lambda}(X,Y)$ with $v_{\mathfrak{p}}(\mathfrak{c}_{i}) = 0$ and $A_{i} \in \Mat_{n \times n}(\mathcal{O}_{\mathfrak{p}})$
for $1 \leq i \leq r$.
\item The matrices $A_{1},\dotsc,A_{r}$ form an $\mathcal{O}_{\mathfrak{p}}$-basis of $\Hom_{\Lambda_{\mathfrak{p}}}(X_{\mathfrak{p}}, Y_{\mathfrak{p}})$.
\item Let $B_{1},\ldots,B_{r} \in  \Mat_{n \times n}(\mathcal{O}_{\mathfrak{p}})$ be any $\mathcal{O}_{\mathfrak{p}}$-basis of $\Hom_{\Lambda_{\mathfrak{p}}}(X_{\mathfrak{p}}, Y_{\mathfrak{p}})$.
The $\Lambda_{\mathfrak{p}}$-lattices $X_{\mathfrak{p}}$ and $Y_{\mathfrak{p}}$ are isomorphic
if and only if there exists a tuple $(b_{1},\dotsc,b_{r}) \in (\mathcal{O}/\mathfrak{p})^{r}$ such that
$\det_{\mathcal{O}/\mathfrak{p}}(b_{1}\overline{B}_{1} + \dotsb + b_{r} \overline{B}_{r}) \neq 0$.
Moreover, if such a tuple exists then an isomorphism is given by $b_{1} B_{1} + \dotsb + b_{r} B_{r}$.
\end{enumerate}
\end{lemma}

\begin{proof}
(a)
By assumption we have $\mathcal{O}_{\mathfrak{p}} \mathfrak{a}_{i} = \mathcal{O}_{\mathfrak{p}} \mathfrak{b}_{i} = \mathcal{O}_\mathfrak{p}$.
Thus
\[
X_{\mathfrak{p}} = \mathcal{O}_\mathfrak{p} X = \bigoplus_{i} \mathcal{O}_{\mathfrak{p}} \alpha_{i}
\quad \textrm{ and } \quad
Y_{\mathfrak{p}} = \mathcal{O}_\mathfrak{p} Y = \bigoplus_{i} \mathcal{O}_{\mathfrak{p}} \beta_{i}.
\]  

(b) Note that $f \in \Hom_{A}(V, W)$ satisfies
$f \in \Hom_{\Lambda_{\mathfrak{p}}}(X_{\mathfrak{p}}, Y_{\mathfrak{p}})$ if and only if
$f(X_{\mathfrak{p}}) \subseteq Y_{\mathfrak{p}}$. By part (a) and the identification \eqref{eq:identify-homgroups},
this is in turn equivalent to $f \in \Mat_{n \times n}(\mathcal{O}_{\mathfrak{p}})$.

(c) From Lemma \ref{lem:homsat}(a) and \eqref{eq:Hom-O-containment} (with $m=n$) we have
\[
\Hom_{\Lambda}(X, Y) = \Hom_{A}(V, W) \cap \Hom_{\mathcal{O}}(X, Y)
\subseteq \Hom_{\mathcal{O}}(X, Y) = \bigoplus_{1 \leq i , j \leq n} \mathfrak{a}_{i}^{-1}\mathfrak{b}_{j} e_{ij}.
\]
Hence by the assumptions on the coefficient ideals and the identification \eqref{eq:identify-homgroups} we have that 
$\Hom_{\Lambda}(X, Y)$ is a subset of $\Mat_{n \times n}(\mathcal{O}_{\mathfrak{p}})$.
Let $((\mathfrak{c}_{i}), (A_{i}))_{1 \leq i \leq r}$ be any pseudo-basis of $\Hom_{\Lambda}(X, Y)$ and
let $\pi \in \mathfrak{p}\setminus \mathfrak{p}^{2}$ be any uniformizer.
Then
\[
\mathfrak{c}_{i} A_{i} = (\mathfrak{c}_{i} / \pi^{v_{\mathfrak{p}}(\mathfrak{c}_{i})})(\pi^{v_{\mathfrak{p}}(\mathfrak{c}_{i})} A_{i}) 
\subseteq \Hom_{\Lambda}(X, Y) \subseteq \Mat_{n\times n}(\mathcal{O}_{\mathfrak{p}}).
\]
Since $\mathfrak{c}_{i} / \pi^{v_\mathfrak{p}(\mathfrak{c}_{i})}$ has $\mathfrak{p}$-adic valuation $0$, we thus have that
every entry of $ \pi^{v_{\mathfrak{p}}(\mathfrak{c}_{i})} A_{i}$ has $\mathfrak{p}$-adic valuation $\geq 0$, that is,
$ \pi^{v_\mathfrak{p}(\mathfrak{c}_{i})}A_{i} \in \Mat_{n \times n}(\mathcal{O}_{\mathfrak{p}})$.
Moreover, $((\mathfrak c_i/\pi^{v_{\mathfrak{p}}(\mathfrak{c}_{i})}), (\pi^{v_{\mathfrak{p}}(\mathfrak{c}_{i})} A_i))_{1 \leq i \leq r}$
is also a pseudo-basis of $\Hom_\Lambda(X, Y)$.
Hence by making the appropriate substitution, we can and do assume without loss of generality that the 
pseudo-basis $((\mathfrak{c}_{i})_{i}, (A_{i})_{i})_{1 \leq i \leq r}$ has the desired properties.

(d) By \cite[(3.18)]{Reiner2003} and the identification of
$\mathcal{O}_{\mathfrak{p}} \Hom_{\Lambda}(X, Y)$ with $\mathcal{O}_{\mathfrak{p}} \otimes_{\mathcal{O}} \Hom_{\Lambda}(X, Y)$,
we have
\[
\Hom_{\Lambda_{\mathfrak{p}}}(X_{\mathfrak{p}}, Y_{\mathfrak{p}})
= \mathcal{O}_{\mathfrak{p}} \Hom_{\Lambda}(X, Y)
= \bigoplus_{i=1}^{r} \mathcal{O}_{\mathfrak{p}} \mathfrak{c}_{i} A_{i}
= \bigoplus_{i=1}^{r} \mathcal{O}_{\mathfrak{p}} A_{i}.
\]

(e) The two $\Lambda_{\mathfrak{p}}$-lattices $X_{\mathfrak{p}}$ and $Y_{\mathfrak{p}}$ are isomorphic if and only if
$\Hom_{\Lambda_{\mathfrak{p}}}(X_{\mathfrak{p}}, Y_{\mathfrak{p}})$ contains an invertible element. 
By part (b), $M \in \Hom_{\Lambda_{\mathfrak{p}}}(X_{\mathfrak{p}}, Y_{\mathfrak{p}})$ is invertible if and only if
$\det_{\mathcal{O}_{\mathfrak{p}}}(M) \in \mathcal{O}_{\mathfrak{p}}^{\times}$.
This is in turn equivalent to $\det_{\mathcal{O}_{\mathfrak{p}}}(M) \bmod{\mathfrak{p}} \neq 0$ in
$\mathcal{O}_{\mathfrak{p}}/\mathfrak{p}\mathcal{O}_{\mathfrak{p}} \cong \mathcal{O}/\mathfrak{p}$.
The claim now follows from the observation that  $\det_{\mathcal{O}_{\mathfrak{p}}}(M) \bmod{\mathfrak{p}} = \det_{\mathcal{O}/\mathfrak{p}}(\overline{M})$
together with the hypothesis on $B_{1}, \ldots, B_{r}$.
\end{proof}

To test whether $X_{\mathfrak{p}}$ and $Y_{\mathfrak{p}}$ are isomorphic as $\Lambda_{\mathfrak{p}}$-lattices,
and to compute an isomorphism if it exists, we can thus proceed as follows:

\begin{enumerate}
\item
Adjust the pseudo-bases of $X$ and $Y$ such that the coefficient ideals have zero $\mathfrak{p}$-adic valuation.
For example, if $\pi \in \mathfrak{p} \setminus \mathfrak{p}^{2}$ is any uniformizer, then
\[
((\pi^{v_{\mathfrak{p}}(\mathfrak a_{i})}\alpha_{i})_{1 \leq i \leq n}, (\mathfrak{a}_{i}/\pi^{v_{\mathfrak{p}}(\mathfrak{a}_{i})})_{1 \leq i \leq n})
\]
is a pseudo-basis of $X$ of the required form (similarly for $Y$).
\item
Compute a pseudo-basis $((\mathfrak{c}_{i}), (A_{i}))_{1 \leq i \leq r}$ of $\Hom_{\Lambda}(X, Y)$
using the algorithm of \S \ref{subsec:comp-homgroups} and adjust it such that the
coefficient ideals have zero $\mathfrak{p}$-adic valuation (as in the proof of Lemma~\ref{lem:padichom}(c)).
\item
Reduce the matrices $A_{1},\dotsc,A_{r}$ modulo $\mathfrak{p}$ to obtain
    $\overline{A}_{1},\dotsc,\overline{A}_{r} \in \Mat_{n \times n}(\mathcal{O}/\mathfrak{p})$.
For every tuple $(a_{1},\dotsc,a_{r}) \in (\mathcal{O}/\mathfrak{p})^{r}$ test whether
$\det_{\mathcal{O}/\mathfrak{p}}(a_{1} \overline{A}_{1} + \dotsb + a_{r} \overline{A}_{r}) \neq 0$.
If such a tuple exists then $X_{\mathfrak{p}}$ and $Y_{\mathfrak{p}}$ are isomorphic as $\Lambda_{\mathfrak{p}}$-lattices
and an isomorphism is given by $a_{1} A_{1} + \dotsb + a_{r} A_{r}$.
Otherwise $X_{\mathfrak{p}}$ and $Y_{\mathfrak{p}}$ are not isomorphic as $\Lambda_{\mathfrak{p}}$-lattices.
\end{enumerate}

%Let $G$ be any finite group and let $p$ be a rational prime. The algorithm above has been implemented for $\Z[G]$-lattices $M, N$ contained in any finitely generated $\Q[G]$-modules $V$ and $W$. The \textsc{Magma} \cite{Wieb1997} code is available on the webpage of the first named author.

\subsection{Probabilistic isomorphism testing}\label{subsec:prob}

The previous two approaches to isomorphism testing reduced the problem to testing the vanishing of potentially
$\#(\mathcal{O}/\mathfrak{p})^{r}$ determinants of elements in $\Mat_{n \times n}(\mathcal{O}/\mathfrak{p})$ (if $X_{\mathfrak{p}}$ and $Y_{\mathfrak{p}}$ are not isomorphic as
$\Lambda_{\mathfrak{p}}$-lattices, then one really needs $\#(\mathcal{O}/\mathfrak{p})^{r}$ determinant computations).
In particular, if either $r$ or $\#(\mathcal{O}/\mathfrak{p})$ is large, this is a rather time-consuming step.
We now connect the problem to polynomial identity testing, which allows us to lower the number of determinant computations in certain situations 
(also see Remark \ref{rmk:when-prob-method-is-useful}).

Let $A_{1},\dotsc,A_{r} \in \Mat_{n \times n}(\mathcal{O}_{\mathfrak{p}})$ be an $\mathcal{O}_{\mathfrak{p}}$-basis of
$\Hom_{\Lambda_{\mathfrak{p}}}(X_{\mathfrak{p}}, Y_{\mathfrak{p}})$ as given by Lemma~\ref{lem:padichom}.
Let $M \mapsto \overline{M}$ denote the canonical projection map $\Mat_{n \times n}(\mathcal{O}_{\mathfrak{p}}) \to \Mat_{n \times n}(\mathcal{O}/\mathfrak{p})$.
Let $T_{1},\dotsc,T_{r}$ be indeterminates and consider the polynomial
\[
f := \det(T_{1} \overline{A}_{1} + \dotsb + T_{r} \overline{A}_{r} ) \in (\mathcal{O}/\mathfrak{p})[T_{1},\dotsc,T_{r}],
\]
which is of total degree $\leq n$.

\begin{lemma}
The two $\Lambda_{\mathfrak{p}}$-lattices $X_{\mathfrak{p}}$ and $Y_{\mathfrak{p}}$ are isomorphic if and only if $f$ as defined above is not the zero-polynomial.
\end{lemma}

\begin{proof}
Suppose that $X_{\mathfrak{p}}$ and $Y_{\mathfrak{p}}$ are isomorphic as $\Lambda_{\mathfrak{p}}$-lattices.
Then by Lemma \ref{lem:padichom}(e) there exist  $a_{1},\dotsc,a_{r} \in \mathcal{O}/\mathfrak{p}$ such that $f(a_{1},\dotsc,a_{r}) \neq 0$ and so 
in particular $f$ is not the zero-polynomial.
Suppose conversely that $f$ is not the zero-polynomial.
Then there exists a finite extension $\mathbb{F}$ of $\mathcal{O}/\mathfrak{p}$ and $b_{1},\ldots,b_{r} \in \mathbb{F}$ such that
$f(b_{1}, \dots, b_{r}) \neq 0$.
Let $S$ be a discrete valuation ring with maximal ideal $\mathfrak{P}$ such that $\mathcal{O}_{\mathfrak{p}} \subseteq S$ and
$S/\mathfrak{P} \cong \mathbb{F}$.
Since $A_{1},\dotsc,A_{r}$ is also an $S$-basis of $\Hom_{S\Lambda_{\mathfrak{p}}}(S X_{\mathfrak{p}}, S Y_{\mathfrak{p}})$,
Lemma \ref{lem:padichom}(e) applied to $S X_{\mathfrak{p}}$ and $S Y_{\mathfrak{p}}$
shows that $b_{1} A_{1} + \dotsb + b_{r} A_{r}$ is an $S\Lambda$-isomorphism $SX_{\mathfrak{p}} \to S Y_{\mathfrak{p}}$.
Hence $X_{\mathfrak{p}}$ and $Y_{\mathfrak{p}}$ are isomorphic as $\Lambda_{\mathfrak{p}}$-lattices
by \cite[(30.25)]{curtisandreiner_vol1}.
\end{proof}

To test whether or not $f$ is the zero-polynomial, we shall make use of the following classical result
(see \cite{Schwartz:1980:FPA:322217.322225,MR575692}).

\begin{theorem}[Schwartz--Zippel lemma]
Let $\mathbb{F}$ be a finite field. For a non-zero polynomial $g \in \mathbb{F}[T_{1},\dotsc,T_{r}]$
of total degree $n < \#\mathbb{F}$ we have
\[
\frac{ \# \{ (a_{1},\dotsc,a_{r}) \in \mathbb{F}^{r} \mid g(a_{1},\dotsc,a_{r}) = 0 \} }{\#\mathbb{F}^{r}} \leq \frac{n}{\#\mathbb{F}}.
\]
%In particular, if $v_{1},\dotsc,v_{r} \in \mathbb{F}^{r}$ are chosen uniformly and
%$f(v_{i}) = 0$ for $1 \leq i \leq r$, then the probability that $f$ is non-zero is at most $(n/\#\mathbb{F})^{r}$.
\end{theorem}

Since $f = 0$ in $\mathbb{F}[T_{1},\dotsc,T_{r}]$ is equivalent to $f = 0$ in $\mathbb{F}'[T_{1},\dotsc,T_{r}]$
for every extension $\mathbb{F}'$ of $\mathbb{F}$, the condition on the total degree is not an actual restriction, as we can just extend scalars if necessary.
We now formulate a probabilistic version of the isomorphism test given in \S \ref{subsec:globhom}.
Let $1 > \varepsilon > 0$ be some chosen error bound.
The following algorithm to test whether $X_{\mathfrak{p}}$ and $Y_{\mathfrak{p}}$ are isomorphic as
$\Lambda_{\mathfrak{p}}$-lattices is of Monte Carlo type, in the sense described in step (iv):

\begin{enumerate}
\item
Compute an $\mathcal{O}_{\mathfrak{p}}$-basis $A_{1},\dotsc,A_{r}$ of 
$\Hom_{\Lambda_{\mathfrak{p}}}(X_{\mathfrak{p}}, Y_{\mathfrak{p}})$ as in \S \ref{subsec:globhom}.
\item
Set $f := \det(T_{1} \overline{A}_{1} + \dotsb + T_{r} \overline{A}_{r}) \in
(\mathcal{O}/\mathfrak{p})[T_{1},\dotsc,T_{r}]$.
\item
Choose $l, k \in \Z_{\geq 1}$ such that $\# (\mathcal{O}/\mathfrak{p})^{l} > n$ and $(n/\#(\mathcal{O}/\mathfrak{p})^{l})^{k} < \varepsilon$.
Let $\mathbb{F}$ be the degree $l$ extension of $\mathcal{O}/\mathfrak{p}$.
\item
Choose $v_{1},\dotsc,v_{k} \in \mathbb{F}^{r}$ uniformly distributed.
For every $i$ in the range $1 \leq i \leq k$ test whether $f(v_{i}) \neq 0$.
If such an $i$ exists then $X_{\mathfrak{p}}$ and $Y_{\mathfrak{p}}$ are isomorphic as $\Lambda_{\mathfrak{p}}$-lattices.
Otherwise, the probability that $X_{\mathfrak{p}}$ and $Y_{\mathfrak{p}}$ are not isomorphic as $\Lambda_{\mathfrak{p}}$-lattices
is at least $1-\varepsilon$.
\end{enumerate}

\begin{remark}\label{rmk:when-prob-method-is-useful}
While the Monte Carlo nature of this algorithm makes it useless if it needs to be shown that $X_{\mathfrak{p}}$ and $Y_{\mathfrak{p}}$ are not isomorphic, there are applications where it can significantly speed up computations.
Assume that we are given $\Lambda$-lattices $X,Y_{1},\dotsc,Y_{m}$ and we know that $X_{\mathfrak{p}}$ must be isomorphic to one 
of $Y_{1,\mathfrak{p}}, \dotsc, Y_{m,\mathfrak{p}}$
(for example, these could be representatives for the isomorphism classes of $\Lambda_{\mathfrak{p}}$-lattices).
Then using the probabilistic algorithm to test $X_{\mathfrak{p}} \cong Y_{i,\mathfrak{p}}$, $1 \leq i \leq m$, with some small $\varepsilon$, we can quickly find the $Y_{i,\mathfrak{p}}$ that is isomorphic to $X_{\mathfrak{p}}$
(it will not be necessary to prove directly that $X_{\mathfrak{p}}$ is not isomorphic to some specific $Y_{j,\mathfrak{p}}$). 
\end{remark}

\subsection{Reduction to testing whether a lattice localized at $\mathfrak{p}$ is free}\label{subsec:redn-to-local-freeness}
If $Y=\Lambda^{(k)}$ for some $k \in \Z_{\geq 1}$ then testing whether $X_{\mathfrak{p}}$ and $Y_{\mathfrak{p}}$
are isomorphic as $\Lambda_{\mathfrak{p}}$-lattices is of course equivalent to testing whether $X_{\mathfrak{p}}$
is free of rank $k$ over $\Lambda_{\mathfrak{p}}$.
An algorithm for computing a $\Lambda_{\mathfrak{p}}$-basis of $X_{\mathfrak{p}}$ (if one exists) is given by Bley and Wilson
in \cite[\S 4.2]{Bley2009}. 
The basic idea is to reduce modulo $\mathfrak{p}$ and then reduce again modulo the Jacobson radical of
$\Lambda_{\mathfrak{p}}/\mathfrak{p}\Lambda_{\mathfrak{p}}$, find a basis over the resulting associative algebra over a finite field
(if one exists), and then lift this basis using Nakayama's lemma (twice). 
If $X_{\mathfrak{p}}$ is free over $\Lambda_{\mathfrak{p}}$ then the lifted elements form a basis. 
Otherwise, either it was not possible to find a basis modulo the Jacobson radical of
$\Lambda_{\mathfrak{p}}/\mathfrak{p}\Lambda_{\mathfrak{p}}$, or the lifted elements do not form a basis.
Since it is straightforward to check whether a given set of elements forms a basis, this gives an algorithm to test
whether $X_{\mathfrak{p}}$ is free and compute a basis if so.
This method is in general faster than the others presented in this section because it does not require an expensive search step.

We now describe how to use this algorithm to give a general isomorphism testing algorithm for localized lattices. 
If one of $X_{\mathfrak{p}}$ and $Y_{\mathfrak{p}}$ is free over $\Lambda_{\mathfrak{p}}$ and the other is not,
then the above algorithm can be used to show that they are not isomorphic as $\Lambda_{\mathfrak{p}}$-lattices.
For the general case we use Proposition \ref{prop:hom-free-over-end} with $\Lambda=\Lambda_{\mathfrak{p}}$,
$X=X_{\mathfrak{p}}$ and $Y=Y_{\mathfrak{p}}$. 
This says that $X_{\mathfrak{p}}$ and $Y_{\mathfrak{p}}$ are isomorphic over $\Lambda_{\mathfrak{p}}$
if and only if the $\End_{\Lambda_{\mathfrak{p}}}(Y_{\mathfrak{p}})$-lattice
$\Hom_{\Lambda_{\mathfrak{p}}}(X_{\mathfrak{p}}, Y_{\mathfrak{p}})$ is free of rank $1$ and every free generator of
$\Hom_{\Lambda_{\mathfrak{p}}}(X_{\mathfrak{p}}, Y_{\mathfrak{p}})$
over $\End_{\Lambda_{\mathfrak{p}}}(Y_{\mathfrak{p}})$ is an isomorphism. 
Also note that any $f \in \Hom_{\Lambda_{\mathfrak{p}}}(X_{\mathfrak{p}}, Y_{\mathfrak{p}})$
is an $\Lambda_{\mathfrak{p}}$-isomorphism if and only if it is an $\mathcal{O}_{\mathfrak{p}}$-isomorphism.
To test whether $X_{\mathfrak{p}}$ and $Y_{\mathfrak{p}}$ are isomorphic as $\Lambda_{\mathfrak{p}}$-lattices,
and to compute an isomorphism if it exists, we can thus proceed as follows:

\begin{enumerate}
\item Compute $\Hom_{\Lambda_{\mathfrak{p}}}(X_{\mathfrak{p}}, Y_{\mathfrak{p}})$ and 
$\End_{\Lambda_{\mathfrak{p}}}(Y_{\mathfrak{p}})$ using the methods described in \S \ref{subsec:globhom}.
\item Use the algorithm of \cite[\S 4.2]{Bley2009} outlined above to check if 
$\Hom_{\Lambda_{\mathfrak{p}}}(X_{\mathfrak{p}}, Y_{\mathfrak{p}})$ is free over 
$\End_{\Lambda_{\mathfrak{p}}}(Y_{\mathfrak{p}})$ and compute a free generator $f$ if so. 
If $\Hom_{\Lambda_{\mathfrak{p}}}(X_{\mathfrak{p}}, Y_{\mathfrak{p}})$ is not free over 
$\End_{\Lambda_{\mathfrak{p}}}(Y_{\mathfrak{p}})$
then $X_{\mathfrak{p}}$ and $Y_{\mathfrak{p}}$ are not isomorphic as $\Lambda_{\mathfrak{p}}$-lattices.
\item Check whether $f:X_{\mathfrak{p}} \rightarrow Y_{\mathfrak{p}}$ is an $\mathcal{O}_{\mathfrak{p}}$-isomorphism.
If so, then it is an isomorphism of $\Lambda_{\mathfrak{p}}$-lattices. 
Otherwise, $X_{\mathfrak{p}}$ and $Y_{\mathfrak{p}}$ are not isomorphic as $\Lambda_{\mathfrak{p}}$-lattices.
\end{enumerate}

\subsection{Implementation of algorithms}\label{subsec:implementation}

Let $G$ be any finite group. The algorithms of \S \ref{subsec:globhom}, \S \ref{subsec:prob} and \S \ref{subsec:redn-to-local-freeness} have been implemented for $\Z[G]$-lattices $X$ and $Y$ contained in any finitely generated $\Q[G]$-modules $V$ and $W$. 
Explicitly, for a given rational prime $p$, these implementations check whether the localizations $X_{p}$ and $Y_{p}$
are isomorphic over $\Z_{(p)}[G]$.
Moreover, the implementations of the algorithms of \S \ref{subsec:globhom} and \S \ref{subsec:redn-to-local-freeness} compute an isomorphism if one exists. 
The \textsc{Magma} \cite{Wieb1997} code is available on the webpage of the first named author.

\begin{remark}
Isomorphism testing for localized lattices is one of the main ingredients
needed to compute a set of representatives of the isomorphism classes of full rank $\Lambda_{\mathfrak{p}}$-lattices of a fixed $A$-module $V$, or equivalently, of full rank $\Lambda_{\mathfrak{p}}$-sublattices of a fixed $\Lambda_{\mathfrak{p}}$-lattice.
The basic idea is to recursively compute maximal sublattices until no new
isomorphism class is found; this goes back to Plesken~\cite{Plesken1974}
(also see \cite[Algorithm 13.7]{Hofmann2016PhD}).
We have used this algorithm together with the algorithms described in this section to determine
a set of representatives of the isomorphism classes of 
full rank $\Z_{(2)}[A_{4}]$-lattices of $\Q[A_{4}]$, where $A_{4}$ is the alternating group on $4$ letters.
In total there are $163$ isomorphism classes, for which the algorithm requires approximately $84\,000$ isomorphism tests.
The average running times for a single isomorphism test using the algorithms
from \S \ref{subsec:globhom}, \S \ref{subsec:prob} (with error bound
$\varepsilon = 2^{-20}$) and \S \ref{subsec:redn-to-local-freeness} are 
$0.0500$ seconds, $0.0087$ seconds and $0.0063$ seconds, respectively.
All computations were performed using \textsc{Magma}~\cite{Wieb1997} V2.22-3 and a single core of a
Intel Xeon CPU E5-2643 v3 @ 3.40GHz.
\end{remark}

\section{Reducing the number of final tests}\label{sec:cut-number-of-tests}

The number of tests in step~(j) of Algorithm~\ref{alg:find-iso} can be enormous.
We now describe an ad hoc method similar to those outlined in \cite[\S 2]{Bley1997} and \cite[\S 7]{Bley2008}
to reduce the number of tests required.
For simplicity, we assume that we are in the case $K=\Q$ and $\mathcal{O}=\Z$.
The method described here generalizes to other cases, and besides, we can reduce to the case $K=\Q$ without loss of generality because 
determining whether $X$ and $Y$ are isomorphic as $\Lambda$-lattices does not depend on whether we view $A$ as a $K$-algebra or a $\Q$-algebra
(though there may be a trade-off in computational cost).

The idea is based on the following simple observation.
Let $d \in \Z_{\geq 1}$ be the $\Z$-rank of $X$ and $Y$ and let $\Omega_{X}$ and $\Omega_{Y}$ be $\Z$-bases of $\mathcal{M}X$ and $\mathcal{M}Y$, respectively.
Denote by $M_{X}, M_{Y} \in \Mat_{d \times d}(\Z)$ basis matrices of $X$ and $Y$ with respect to $\Omega_{X}$ and $\Omega_{Y}$.
Now if $h \colon \mathcal{M} X \to \mathcal{M}Y$ is any $\Z$-linear map with basis matrix $M$ with respect to $\Omega_{X}$ and $\Omega_{Y}$, then $h(X) \subseteq Y$ if and only if $M_{X} M M_{Y}^{-1} \in \Mat_{d \times d}(\Z)$.
We will show how appropriate choices of bases and basis matrices can help us to reduce the number of tests in step~(j).

Recall from \S \ref{subsec:res-iso-max} that we have a decomposition $\mathcal{M} = \mathcal{M}_{1} \oplus \dotsb \oplus \mathcal{M}_{r}$
which induces decompositions 
\begin{align}\label{eq:dec}
\mathcal{M}X = \mathcal{M}_{1} X \oplus \dotsb \oplus \mathcal{M}_{r}X
\quad \textrm{ and } \quad
\mathcal{M}Y = \mathcal{M}_{1} Y \oplus \dotsb \oplus \mathcal{M}_{r}Y.
\end{align}
From the previous steps of Algorithm~\ref{alg:find-iso}, 
we have isomorphisms $f_{i} \colon \mathcal{M}_{i} X \to \mathcal{M}_{i}Y$ of $\mathcal{M}_{i}$-lattices
and finite subsets $U_{i} \subseteq \Aut_{\mathcal{M}_{i}}(\mathcal{M}_{i} Y)$ for $1 \leq i \leq r$.
We need to check whether there exists $(g_{1},\dotsc,g_{r}) \in U_{1} \times \dotsb \times U_{r}$ such that $(\sum_{i=1}^{r} (g_{i} \circ f_{i})))(X) \subseteq Y$.

Let $d_i \in \Z_{\geq 1}$ denote the $\Z$-rank of $\mathcal{M}_{i}X$, which is also equal to the $\Z$-rank of $\mathcal{M}_{i} Y$.
Now we choose $\Z$-bases $\Omega_{X}$ and $\Omega_{Y}$ of $\mathcal{M}X$ and $\mathcal{M}Y$ adapted to the decompositions in~(\ref{eq:dec}).
For isomorphisms $f_{i} \colon \mathcal{M}_{i} X \to \mathcal{M}_{i} Y$ the matrix representing $\sum_{i=1}^{r}f_{i}$ is a block matrix of the form
$\operatorname{diag}(M(f_{1}),\dotsc,M(f_{r}))$ with $M(f_{i}) \in \Mat_{d_{i} \times d_{i}}(\Z)$.
Similarly, for automorphisms $g_{i} \colon \mathcal{M}_{i} Y \to \mathcal{M}_{i} Y$ the matrix representing $\sum_{i=1}^{r} g_{i}$ is a block diagonal matrix
of the form $\operatorname{diag}(M(g_{1}), \dotsc, M(g_{r}))$ with $M(g_{i}) \in \GL_{d_{i}}(\Z)$.
Hence with respect to $\Omega_{X}$ and $\Omega_{Y}$, the morphism $h := \sum_{i=1}^{r} g_{i} \circ f_{i}$ is given by
\[
M = \operatorname{diag}(M(g_{1})M(f_{1}), \dotsc,M(g_{r})M(f_{r})).
\]
Now let $M_{X}$ and $M_{Y}$ be upper triangular basis matrices of $X$ and $Y$ with respect to $\Omega_{X}$ and $\Omega_{Y}$. 
Then $h$ satisfies $h(X) \subseteq Y$ if and only if $M_{X} M M_{Y}^{-1} \in \Mat_{d \times d}(\Z)$.
Setting $\tilde{M}_{Y} := M_{Y}^{-1} \in \GL_{d}(\Q)$, the matrix $M_{X} M M_{Y}^{-1}$ is equal to
\[
\small
\left(
\begin{array}{ccccc}
\rright{M_{X}^{(1)}} & & \ast                                    \\ \cline{1-1}
  &                \ddots                \\ \cline{3-3}
  \text{0} & &    \lleft{M_{X}^{(r)}}            \\
\end{array}
\right)
\left(
\begin{array}{ccccc}
\rright{M(g_{1})M(f_{1})} & & 0                                    \\ \cline{1-1}
  &                \ddots                \\ \cline{3-3}
  \text{0} & &    \lleft{M(g_{r})M(f_{r})}            \\
\end{array}
\right)
 \left(
\begin{array}{ccccc}
\rright{\tilde M_{Y}^{(1)}} & & \ast                                    \\ \cline{1-1}
  &                \ddots                \\ \cline{3-3}
  \text{0} & &    \lleft{\tilde M_{Y}^{(r)}}            \\
\end{array}
\right),
\]
with $M_{X}^{(i)} \in \Mat_{d_{i} \times d_{i}}(\Z)$, $\tilde{M}_{Y}^{(i)} \in \GL_{d_{i}}(\Q)$ for $i = 1,\dotsc,r$.
As this product of matrices is equal to 
\[     \left(
\begin{array}{ccccc}
\rright{M_X^{(1)}M(g_1)M(f_1)\tilde M_Y^{(1)}} & & \ast                                    \\ \cline{1-1}
  &                \ddots                \\ \cline{3-3}
  \text{0} & &    \lleft{M_X^{(r)}M(g_r)M(f_r) \tilde M_Y^{(r)}}            \\
\end{array}
\right),
\]
we see that the $i$th block on the diagonal depends only on $g_{i}$ and is independent of $g_{j}$ for $j \neq i$.
In particular, if for example we find $g_{1} \in U_{1}$ such that 
\[
M_{X}^{(1)}M(g_{1})M(f_{1})\tilde{M}_{Y}^{(1)} \notin \Mat_{d_{1} \times d_{1}}(\Z),
\]
then we can remove all elements $\{ g_{1} \} \times U_{2} \times \dotsb \times U_{r}$ from the search space.

\section{Computational results}\label{sec:exp}

We have a proof of concept implementation of Algorithm \ref{alg:find-iso} in \textsc{Magma}~\cite{Wieb1997}, which works in the following situation.
Let $G$ be a finite group, let $A=\Q[G]$ and let $\Lambda=\Z[G]$.
Let $A=A_{1} \oplus \cdots \oplus A_{r}$ be the decomposition of $A$ into indecomposable two-sided ideals and 
let $K_{i}$ denote the center of the simple algebra $A_{i}$.
For each $i$ there is an isomorphism $A_{i} \cong \Mat_{n_{i} \times n_{i}}(D_{i})$ of $\Q$-algebras,
where $D_{i}$ is a skew field with center $K_{i}$.
Suppose that for each $i$, at least one of the following holds: (a) $D_{i}=\Q$, (b) $A_{i}=D_{i}=K_{i}$ (so in particular $n_{i}=1$), or
(c) $D_{i}$ is a quaternion algebra and $n_{i}=1$ (that is, $A_{i}=D_{i}$ is a skew field and $[D_{i}:\Q]=4$).
This condition holds in each of the following cases:
\renewcommand{\labelenumi}{(\roman{enumi})}
\begin{enumerate}
\item $G$ is abelian;
\item $G=S_{n}$, the symmetric group on $n$ letters;
\item $G=\F_{q} \rtimes \F_{q}^{\times}$, where $\F_{q}$ is the finite field with $q \geq 3$ elements and the semidirect product is defined by the natural action (such a group has a unique non-linear irreducible character, which is rationally represented); 
\item $G=Q_{8}, Q_{12}, Q_{8} \times C_{2}$ or $Q_{12} \times C_{2}$, where $Q_{4n}$ is the quaternion group of order $4n$ and $C_{2}$ is the cyclic group of order $2$.
\end{enumerate}
The implementation can decide whether two $\Z[G]$-lattices contained in $\Q[G]$ are isomorphic and, if so, give an explicit isomorphism.
In practice, the number of final tests required for step (j) of Algorithm \ref{alg:find-iso} is too high in many cases 
(when $G=Q_{12} \times C_{2}$, for example), even if the methods of \S \ref{sec:cut-number-of-tests} are employed.
The code is available on the webpage of the first named author.

The implementation can be used to investigate the Galois module structure of arithmetic objects such as the rings of integers 
and ambiguous ideals of Galois extensions $K/\Q$ with $\Gal(K/\Q) \cong G$.
We note that it is straightforward to realise these lattices as lattices in $\Q[G]$ by finding a normal basis generator of $K/\Q$, which can be done in several ways; for example, one can use the algorithm of Girstmair \cite{MR1706933}.

\subsection{Galois $G$-extensions}\label{subsec:Galois-G-extns}

Let $K$ be a number field and let $G$ be a finite group.
We fix a $G$-extension of $K$, that is, a pair $(L, \varphi)$
consisting of a Galois extension $L/K$ together with a group isomorphism $\varphi \colon G \to \Gal(L/K)$.
In this way one obtains an action of $G$ on $L$.
The classical Normal Basis Theorem implies that $(L,\varphi) \cong K[G]$ as $K[G]$-modules.
Since $\mathcal{O}_{L}$ is a $\Gal(L/K)$-invariant finitely generated torsion-free $\mathcal{O}_{K}$-module, 
the pair $(\mathcal{O}_{L}, \varphi)$ uniquely defines a $\mathcal{O}_{K}[G]$-lattice in $L$.
It is straightforward to see that if $(\mathcal{O}_{L},\varphi)$ is free (resp.\ stably free, locally free) then $(\mathcal{O}_{L},\psi)$ is also free
(resp.\ stably free, locally free) for any choice of isomorphism $\psi : G \rightarrow \Gal(L/K)$.

Henceforth assume that $L/K$ is (at most) tamely ramified. 
Then it is well known that $(\mathcal{O}_{L},\varphi)$ is a locally free $\mathcal{O}_{K}[G]$-lattice of rank $1$
(see \cite{Noether1932}, \cite[I, \S 3]{MR717033} or \cite{MR825142}).
Let $LF_{1}(\mathcal{O}_{K}[G])$ denote the set of isomorphism classes of locally free $\mathcal{O}_{K}[G]$-lattices of rank $1$
and let $[\mathcal{O}_{L}^{\varphi}] \in LF_{1}(\mathcal{O}_{K}[G])$ denote the isomorphism class of $(\mathcal{O}_{L},\varphi)$.
An automorphism $\theta \in \Aut(G)$ induces an action on $LF_{1}(\mathcal{O}_{K}[G])$ such that in particular
$\theta \cdot [ \mathcal{O}_{L}^{\varphi} ] = [ \mathcal{O}_{L}^{\varphi \circ \theta}]$.
We define
\[
\{ [\mathcal{O}_{L}] \} := \{ [\mathcal{O}_{L}^{\psi}] \in LF_{1}(\mathcal{O}_{K}[G]) \mid \psi : G \rightarrow \Gal(L/K) \textrm{ is an isomorphism} \}.
\]
Since any two isomorphisms $G \rightarrow \Gal(L/K)$ differ by an element of $\Aut(G)$, we see that $\{ [\mathcal{O}_{L}] \}$
is the $\Aut(G)$-orbit of $[\mathcal{O}_{L}^{\varphi}]$ in $LF_{1}(\mathcal{O}_{K}[G])$, and this is independent of the choice of $\varphi$.
For $\sigma \in \Gal(L/K)$ let $\iota_{\sigma}$ denote the inner automorphism of $\Gal(L/K)$ defined by $\tau \mapsto \sigma \tau \sigma^{-1}$.
Then it is straightforward to check that the map $\sigma: (\mathcal{O}_{L},\varphi) \rightarrow (\mathcal{O}_{L}, \iota_{\sigma} \circ \varphi)$
is an isomorphism of $\mathcal{O}_{K}[G]$-lattices. 
Thus the action of $\Aut(G)$ on $\{ [\mathcal{O}_{L}] \}$ factors through $\Out(G):=\Aut(G)/\Inn(G)$ where $\Inn(G)$
is the normal subgroup of $\Aut(G)$ consisting of inner automorphisms.

\subsection{Rings of integers of $Q_{8} \times C_{2}$-extensions}\label{subsec:exp1}
Let $G = Q_{8} \times C_{2}$, the direct product of the quaternion group of order $8$ and the cyclic group of order $2$.
Swan \cite{Swan1983} showed there exist $\Z[G]$-lattices that are stably free but not free, 
but that for any group $H$ with $|H|<16$, every stably free $\Z[H]$-lattice is in fact free.
He also showed that $|LF_{1}(\Z[G])|=40$ and that there are $4$ classes in $LF_{1}(\Z[G])$ that are stably free.
We label these classes $\mathcal{C}_{1}, \mathcal{C}_{2}, \mathcal{C}_{3}, \mathcal{C}_{4}$, where
$\mathcal{C}_{1}$ is the class of free $\Z[G]$-lattices of rank $1$ and the lattices contained in the other classes are stably free  but not free.
(Note that the labelling of $\mathcal{C}_{2},\mathcal{C}_{3}$ and $\mathcal{C}_{4}$ is arbitrary,
but the key point is that it will be fixed for the rest of this discussion.)
Based upon these results, Cougnard \cite{cougnard-H8-C2} showed that for each $1 \leq i \leq 4$,
there exist infinitely many $G$-extensions $(L, \varphi)$ of $\Q$ with $[ \mathcal{O}_{L}^{\varphi}] = \mathcal{C}_{i}$.
The group of automorphisms $\Aut(G)$ has order 192 and the quotient group of outer automorphisms $\Out(G)$ has order $48$.
Using either our implementation of Algorithm \ref{alg:find-iso} or the description and discussion of $LF_{1}(\Z[G])$ in \cite[\S 16]{Swan1983},
it can be shown that the action of $\Aut(G)$ on $\{ \mathcal{C}_{1},\mathcal{C}_{2},\mathcal{C}_{3},\mathcal{C}_{4} \}$ has two orbits,
namely $\{ \mathcal{C}_{1} \}$ and $\{ \mathcal{C}_{2}, \mathcal{C}_{3}, \mathcal{C}_{4} \}$.

We now describe how we used our implementation of  Algorithm \ref{alg:find-iso} to verify the numerical examples considered
by Cougnard.
Let $N/\Q$ be a finite Galois extension with $\Gal(N/\Q) \cong Q_{8}$ and fix a choice of 
isomorphism $\varphi \colon Q_{8} \to \Gal(N/\Q)$.
Let $d \in \Q$ such that $\sqrt{d} \notin N$. Then $N(\sqrt{d})/\Q$ is Galois and there is a canonical identification
\begin{equation}\label{eq:galois-direct-prod}
\Gal(N(\sqrt{d})/\Q) = \Gal(N/\Q) \times \Gal(\Q(\sqrt{d})/\Q). 
\end{equation}
Moreover, since $\Aut(C_{2})$ is trivial there is a unique isomorphism
$\theta_{d} \colon C_{2} \rightarrow \Gal(\Q(\sqrt{d})/\Q)$.
Using \eqref{eq:galois-direct-prod} we define an isomorphism 
$\varphi_{d} \colon Q_{8} \times C_{2} \longrightarrow \Gal(N(\sqrt d)/\Q)$ by $\varphi_{d} = \varphi \times \theta_{d}$
(this is the unique isomorphism whose restriction back to the first factor recovers $\varphi$).
Let $[\mathcal{O}_{N(\sqrt{d})}]$ denote the $\Z[G]$-isomorphism class of $(\mathcal{O}_{N(\sqrt{d})}, \varphi_{d})$.
It is interesting to compare such isomorphism classes as $d$ varies while keeping $N$ and $\varphi$ fixed.

In \cite[VIII]{cougnard-H8-C2}, Cougnard considered the number field $N_{1}$ with defining polynomial
\begin{multline*}
  \mathtt{x^8 - x^7 + 62126x^6 - 565081x^5 + 1060385071x^4 - 16366741325x^3}\\
  \mathtt{ + 465279400700x^2 + 7092550941085x + 160472449673155 \in \Q[x]}.
\end{multline*}
We were able to compute the following isomorphism classes of rings of integers: 
\renewcommand{\labelenumi}{(\roman{enumi})}
\begin{enumerate}
\item $\mathcal{C}_{1} = [ \mathcal{O}_{N_{1}(\sqrt{5})} ] = [ \mathcal{O}_{N_{1}(\sqrt{221})}]$, 
\item $\mathcal{C}_{2} = [ \mathcal{O}_{N_{1}(\sqrt{17})}]$,
\item $\mathcal{C}_{3} = [ \mathcal{O}_{N_{1}(\sqrt{13})} ] = [ \mathcal{O}_{N_{1}(\sqrt{21})} ]  = [ \mathcal{O}_{N_{1}(\sqrt{65})} ]$,
\item $\mathcal{C}_{4} = [ \mathcal{O}_{N_{1}(\sqrt{85})} ]$. 
\end{enumerate}
In case (i), our implementation of Algorithm \ref{alg:find-iso} yielded explicit $\Z[G]$-isomorphisms of the rings of integers with $\Z[G]$
and so we obtained explicit normal integral bases
(as the coefficients of the elements generating the normal integral bases are quite large, we do not reproduce them here).
The implementation was also used to check whether any two of the rings of integers listed above are isomorphic or not as $\Z[G]$-lattices,
and thus verified that the isomorphism classes listed above are indeed distinct.
We used two independent methods to check that all the rings of integers above are stably free
and thus do in fact belong to the isomorphism classes $\mathcal{C}_{2}, \mathcal{C}_{3}$ and $\mathcal{C}_{4}$ in cases (ii), (iii) and (iv)
(recall the labelling of these classes is arbitrary but fixed).
Note that a locally free $\Z[G]$-lattice of rank $1$ is stably free if and only if it has trivial class in the 
locally free class group $\Cl(\Z[G])$.
The first method was to use an algorithm of Bley and Wilson \cite{Bley2009} that solves the discrete logarithm problem in $\Cl(\Z[G])$.
This algorithm was implemented in \textsc{Magma}~\cite{Wieb1997} by Bley and the code is available on his website.
The second method was to use the remarkable work of Fr\"ohlich and Taylor that determines
the class in $\Cl(\Z[G])$ of the ring of integers of a tamely ramified $G$-extension 
in terms of the Artin root numbers of the irreducible symplectic characters of $G$ (see \cite[I, \S 6]{MR717033}).
We used the \textsc{Magma} command \texttt{RootNumber} to show that these root numbers are $1$ in all the cases above, which implies that the
classes in $\Cl(\Z[G])$ are trivial.
Therefore the results above are in agreement with those of Cougnard (he also considered similar situations starting with fields other than $N_{1}$,
but we do not consider these here).

Note that once a single representative of each isomorphism class has been found, 
our implementation of Algorithm \ref{alg:find-iso} can be used to check whether any given 
locally free  $\Z[G]$-lattice of rank $1$ is stably free or not and so 
one does not need to apply either of the two methods described above when investigating further examples. 
Moreover, in principle one can use Algorithm \ref{alg:find-iso} to check whether a ring of integers is stably free over $\Z[G]$
without using the above methods at all. Let $(L,\varphi)$ be a $G$-extension of $\Q$. 
Then the Bass Cancellation Theorem \cite[(41.20)]{curtisandreiner_vol2} shows that $(\mathcal{O}_{L}, \varphi)$ is stably free over $\Z[G]$
if and only if $(\mathcal{O}_{L},\varphi) \oplus \Z[G] \cong \Z[G] \oplus \Z[G]$ as $\Z[G]$-modules (this is independent of the choice of $\varphi$).
This of course can be checked by Algorithm \ref{alg:find-iso}, but unfortunately our implementation is restricted to locally free $\Z[G]$-lattices in $\Q[G]$.

We have used our implementation of Algorithm \ref{alg:find-iso} to investigate the distribution of the Galois module structure of 
the rings of integers among all tamely ramified Galois extensions $L/\Q$ with $\Gal(L/\Q) \cong G$ and
$\lvert \disc(\mathcal{O}_{L}) \rvert \leq 10^{40}$. 
(Note that the fields considered in the previous paragraph all satisfy $\lvert \disc(\mathcal{O}_{N_{1}(\sqrt{d})}) \rvert \geq 10^{64}$
because $\lvert \disc(\mathcal{O}_{N_{1}}) \rvert \geq 10^{32}$.)
Since any such field $L$ is a tamely ramified quadratic extension of a tamely ramified Galois extension
$L_{0}/\Q$ with $\Gal(L_{0}/\Q) \cong Q_{8}$ and $\lvert \disc(\mathcal{O}_{L_{0}}) \rvert \leq 10^{20}$, 
we first used algorithms based on class field theory described in~\cite{FHS2019} and implemented in \textsc{Hecke}~\cite{FHHJ2017} to construct all the possible $L_{0}$
(there are 235 such fields).
We then used the same techniques to build appropriate quadratic extensions of these fields.
In total there are $315$ extensions $L/\Q$ with the desired properties (one needs to take care to discard duplicates).
In order to avoid the dependence on the choice of isomorphism $\varphi: G \rightarrow \Gal(L/\Q)$,
we only determine whether $[\mathcal{O}_{L}^{\varphi}]$ lies in $\mathcal{C}_{1}$, in $\mathcal{C}_{2} \cup \mathcal{C}_{3} \cup \mathcal{C}_{4}$,
or in neither of these. In other words, we determine whether $\mathcal{O}_{L}$ is free, stably free but not free or not stably free over $\Z[\Gal(L/\Q)]$.
Of the $315$ rings of integers $\mathcal{O}_{L}$ under consideration, $80$ are free, $1$ is stably free but not free and $234$ are not stably free.
The one stably free but not free example is the ring of integers of the number field $L$ with defining polynomial
\begin{multline*}
\mathtt{x^{16} + 11x^{15} - 603x^{14} - 3827x^{13} + 145692x^{12} + 266691x^{11} - 16993778x^{10} + 30104389x^{9}} \\
\mathtt{+ 898058760x^{8} - 4356130039x^{7} - 11656785671x^{6} + 135624739908x^{5} - 369009691593x^{4}} \\
\mathtt{+ 364395270692x^{3} + 8335437012x^{2} - 166048630160x + 22344148336 \in \Q[x]}
\end{multline*}
and discriminant
\[ 
\mathtt{
9486970677311569898939510744199462890625 = 3^{12} \cdot 5^{12} \cdot 11^{12} \cdot 13^{12}.
}
\]
Since our table of number fields is complete with respect to the given absolute discriminant bound,
$L$ is in fact the number field of smallest absolute discriminant with the property that $L/\Q$ is Galois with $\Gal(L/\Q) \cong G$
and that $\mathcal{O}_{L}$ is stably free but not free over $\Z[\Gal(L/\Q)]$ (note that this forces $L/\Q$ to be tamely ramified).

\subsection{Ambiguous ideals}
We first recall some general properties of ambiguous ideals following Ullom \cite[Chapter I]{Ullom1969}.
Let $K$ be a number field and let $G$ be a finite group.
Let $(L,\varphi)$ be a tamely ramified $G$-extension of $K$ and let $\mathfrak{a}$ be an ambiguous ideal of $\mathcal{O}_{L}$,
that is, an ideal that is invariant under the action of $G$ (note that this property does not depend on the choice of $\varphi$).
Then $(\mathfrak{a},\varphi)$ uniquely defines an $\mathcal{O}_{K}[G]$-lattice in $L$.
Since $L/K$ is tamely ramified, $(\mathfrak{a},\varphi)$ is locally free 
and the observations made in \S \ref{subsec:Galois-G-extns} also apply in this setting.
For a maximal ideal $\mathfrak{p}$ of $\mathcal{O}_{K}$ decomposing as
$\mathfrak{p} \mathcal{O}_{L} = (\mathfrak{P}_{1} \dotsm \mathfrak{P}_{g})^{e}$ we set
$\psi(\mathfrak{p}) = \mathfrak{P}_{1} \dotsm \mathfrak{P}_{g}$.
We have the following classification.
\begin{itemize}
\item
The ideal $\psi(\mathfrak{p})$ is ambiguous and the set $\{\psi(\mathfrak{p}) \mid \text{$\mathfrak{p}$ a maximal ideal of
$\mathcal{O}_{K}$}\}$
is a free $\Z$-basis of the abelian group of ambiguous ideals of $\mathcal{O}_{L}$.
\item
Every ambiguous ideal $\mathfrak{a}$ of $\mathcal{O}_{L}$ can be uniquely written in the form
$\mathfrak{a} = \mathfrak{a}_{0} \mathfrak{b}$, with $\mathfrak{b}$ an ideal of $\mathcal{O}_{K}$ and
\[
\mathfrak{a}_{0} = \psi(\mathfrak{p}_{1})^{a_{1}} \dotsm \psi(\mathfrak{p}_{t})^{a_{t}}, \quad 0 \leq a_{i} < e_{i},
\]
where $e_{i} > 1$ is the ramification index of a maximal ideal of $\OL$ dividing $\mathfrak{p}_{i}$.
The ideal $\mathfrak{a}_{0}$ is called a primitive ambiguous ideal.
\end{itemize}
If $K=\Q$ then $\mathcal{O}_{K} = \Z$ is a principal ideal domain and so every ambiguous ideal
$\mathfrak{a}_{0} \mathfrak{b}$ with $\mathfrak{a}_{0}$ primitive and $\mathfrak{b}$ an ideal of $\Z$
is isomorphic to $\mathfrak{a}_{0}$ as a $\Z[\Gal(L/\Q)]$-module.
Thus when investigating the possible Galois module structure of ambiguous ideals in this situation, we can restrict to primitive ambiguous ideals.

In the sequel, we extend $\psi$ to all non-zero fractional ideals of $\mathcal{O}_{K}$ so that 
if $a_{1}, \ldots, a_{t} \in \Z$ and $\mathfrak{p}_{1}, \ldots, \mathfrak{p}_{t}$ are maximal ideals of $\mathcal{O}_{K}$ then
$\psi(\mathfrak{p}_{1}^{a_{1}} \dotsm \mathfrak{p}_{t}^{a_{t}})=\psi(\mathfrak{p}_{1})^{a_{1}} \dotsm \psi(\mathfrak{p}_{t})^{a_{t}}$.

\subsection{Ambiguous ideals for a fixed $Q_{8} \times C_{2}$-extension}

We now specialise to the case $G = Q_{8} \times C_{2}$ and $K = \Q$ as in \S \ref{subsec:exp1}.
Let $N_{1}$ denote the extension of $\Q$ defined in \S \ref{subsec:exp1}.
Let $L_{1} = N_{1}(\sqrt{5})$ and note that $L_{1}/\Q$ is a Galois extension with $\Gal(L_{1}/\Q) \cong G$.
Since the discriminant of $\mathcal{O}_{L_{1}}$ is
\[
\mathtt{3^{12} \cdot 5^{12} \cdot 7^{12} \cdot 11^{12} \cdot 13^{12} \cdot 17^{12}},
\]
the extension $L_{1}/\Q$ is tamely ramified and thus every ambiguous ideal of $\mathcal{O}_{L_{1}}$ is a locally free
$\Z[G]$-lattice.
Moreover, as the ramification indices of the rational primes dividing the discriminant are all equal to $4$, there are $4^{6} = 4096$
primitive ambiguous ideals. In \cite[VIII]{cougnard-H8-C2}, Cougnard showed that $\mathcal{O}_{L_{1}}$ is a free 
$\Z[\Gal(L_{1}/\Q)]$-lattice.
We verified this result with our implementation of Algorithm \ref{alg:find-iso} and also investigated the $\Z[\Gal(L_{1}/\Q)]$-structure of all the primitive ambiguous ideals:
$1024$ are free, $1024$ are stably free but not free and $2048$ are not stably free.
More precisely, for a fixed choice of isomorphism $\varphi \colon G \rightarrow \Gal(L_{1}/\Q)$, all $1024$ stably free but not free primitive ambiguous ideals lie in the same $\Z[G]$-isomorphism class.
Examples of free, stably free but not free and not stably free ambiguous ideals are $\psi(17 \Z)$, $\psi(17^{2} \Z)$ and $\psi(11^{2} \Z)$, respectively.

Now let $L_{2} = N_{1}(\sqrt{221})$. Again, the extension $L_{2}/\Q$ is Galois with $\Gal(L_{2}/\Q) \cong G$.
Since $\mathcal{O}_{L_{2}}$ has the same discriminant as $\mathcal{O}_{L_{1}}$, the extension $L_{2}/\Q$ is also tamely ramified and
there are $4096$ primitive ambiguous ideals whose $\Z[\Gal(L_{2}/\Q)]$-structure is as follows:
$512$ are free, $1536$ are stably free but not free and $2048$ are not stably free
(in particular, $\mathcal{O}_{L_{2}}$ is free, as was first shown by Cougnard \cite[VIII]{cougnard-H8-C2}).
Moreover, for a fixed choice of isomorphism $\varphi \colon G \rightarrow \Gal(L_{2}/\Q)$,
the $1536=3 \cdot 512$ stably free but not free primitive ambiguous ideals are equally distributed among the three isomorphism classes of stably free but not free $\Z[G]$-lattices of rank $1$.
The ambiguous primitive ideals $\psi(5^{2}17^{2}\Z)$, $\psi(5^{2}\Z)$ and $\psi(7^{2}11^{2}\Z)$ are pairwise non-isomorphic 
(for a fixed $\varphi$) and stably free but not free. Moreover, $\psi(11^{2}\Z)$ is not stably free and $\psi(17 \Z)$ is free but not equal to $\mathcal{O}_{L_{1}}$.

\bibliography{locisoiso}
\bibliographystyle{amsalpha}

\end{document}